\documentclass[11pt,letterpaper]{amsart}
\usepackage[letterpaper,margin=1.50in]{geometry}

\usepackage{graphicx}
\usepackage[colorlinks, citecolor=blue, linkcolor=black]{hyperref}
 \usepackage{fourier}

\newtheorem{theorem}{Theorem}[section]

\newtheorem{lemma}[theorem]{Lemma}
\newtheorem{proposition}[theorem]{Proposition}

\newtheorem{remark}[theorem]{Remark}

\theoremstyle{definition}
\newtheorem{definition}[theorem]{Definition}
\newtheorem{example}[theorem]{Example}

\numberwithin{equation}{section}

\begin{document}
\date{}
\title[{Minimal surface system in Euclidean four-space}]
{\large{Minimal surface system in Euclidean four-space}}

\author{Hojoo Lee}
\email{momentmaplee@gmail.com}
\thanks{2010 {\em Mathematics Subject Classification.} 53A10, 49Q05.}

\keywords{Generalized Gauss map, minimal surface system, special Lagrangian equation.}

\maketitle

  \begin{center}
 {An homage to Robert Osserman's book  \emph{\textbf{A Survey of Minimal Surfaces}}}
 \end{center}
 
  \bigskip

\section{Main results}

Osserman \cite{Osserman 1968, Osserman 1969, Osserman 1973} initiated the study of the minimal surface system in arbitrary codimensions. 
In 1977, Lawson and Osserman \cite{Lawson Osserman 1977} found 
fascinating counterexamples to the existence, uniqueness and regularity of solutions to the minimal surface system. 
As a remarkable nonlinear extension of Riemann's removable singularity theorem, Bers' Theorem \cite{Abikoff Corillon Gilman Kra Weinstein 1995, 
Bers 1951, Finn 1953, Osserman 1973} reveals that an isolated singularity of a single valued solution to the minimal surface equation in ${\mathbb{R}}^{3}$
is removable. However, as in Example \ref{punctured example}, the minimal surface system in ${\mathbb{R}}^{4}$ can have solutions with isolated 
non-removable singularities.

Extending Bernstein's Theorem that the only entire minimal graphs in ${\mathbb{R}}^{3}$ are planes, Osserman established that any entire two 
dimensional minimal graph in ${\mathbb{R}}^{4}$ should be \emph{degenerate}, in the sense that its generalized Gauss map (Definition \ref{def 
of generalized gauss map})  lies on a hyperplane of  the complex projective space ${\mathbb{CP}}^{3}$. Landsberg \cite{Landsberg 1992} investigated 
the systems of the first order whose solutions induce minimal varieties. The classical \textbf{Cauchy--Riemann equations} 
$\left(  f_{x},  f_{y} \right) = \left(   g_{y}, -  g_{x}  \right)$ solves the \textbf{minimal surface system} of the second order 
\[
\begin{cases}
 0 =   \left( 1+ {f_{y}}^{2}+{g_{y}}^{2} \right) f_{xx} - 2 \left(  f_{x} f_{y}  +  g_{x} g_{y} \right) f_{xy} +    \left( 1+ {f_{x}}^{2}+{g_{x}}^{2} \right) f_{yy},  \\
 0 =   \left( 1+ {f_{y}}^{2}+{g_{y}}^{2} \right) g_{xx} - 2 \left(  f_{x} f_{y}  +  g_{x} g_{y} \right) g_{xy} +    \left( 1+ {f_{x}}^{2}+{g_{x}}^{2} \right) g_{yy},
\end{cases}
\]
which geometrically means that the two dimensional graph of height functions $f(x,y)$ and $g(x,y)$ is a minimal surface in ${\mathbb{R}}^{4}$. 
A holomorphic curve, or equivalently, special Lagrangian $2$-fold in ${\mathbb{C}}^{2}$ can be identified as a minimal surface in ${\mathbb{R}}^{4}$
whose generalized Gauss map lies on the intersection of two hyperplanes $z_{1}+i z_{2}=0$ and  $z_{3}+i z_{4}=0$ in  ${\mathbb{CP}}^{3}$.
 Generalizing the Cauchy--Riemann equations, we construct the Osserman system of the first order, whose solution graphs become 
 degenerate minimal surfaces in ${\mathbb{R}}^{4}$. 

\begin{theorem}[\textbf{Osserman system as a generalization of Cauchy--Riemann equations}] \label{main one}
Let ${\Sigma}$ be the graph in ${\mathbb{R}}^{4}$ of the pair $\left( f(x, y), g(x,y)  \right)$ of height functions 
defined on $\Omega$
\[
 \Sigma =\left\{\, \mathbf{\Phi}(x,y) = \left(x, y, f(x, y), g(x,y) \right) \in {\mathbb{R}}^{4} \, \vert \, (x,y) \in \Omega \, \right\}.
\]
The induced metric and area element on the surface $\Sigma$ reads
\[
{\mathbf{g}}_{{\Sigma}} = E dx^{2} + 2F dx dy  + G dy^{2}, \quad  dA_{\Sigma} = \omega \, dx \, dy, \quad \omega= \sqrt{EG -F^2},
\]
where the coefficients of the first fundamental form are determined by 
\[
E = {\mathbf{\Phi}}_{x} \cdot {\mathbf{\Phi}}_{x} = 1+ {f_{x}}^{2}+{g_{x}}^{2}, \;
F = {\mathbf{\Phi}}_{x} \cdot {\mathbf{\Phi}}_{y} =   f_{x} f_{y}  +  g_{x} g_{y}, \;
G = {\mathbf{\Phi}}_{y} \cdot {\mathbf{\Phi}}_{y} = 1+ {f_{y}}^{2}+{g_{y}}^{2}.
\]
Assume that  $\left( f(x,y), g(x,y) \right)$ obeys our Osserman system with the coefficient  $\mu \in \mathbb{R}-\{0\}:$
\[
  \begin{bmatrix}   f_x \\ f_y \end{bmatrix}  =  \mu   \begin{bmatrix}  \frac{E}{\omega} & \frac{F}{\omega} 
  \\ \frac{F}{\omega} & \frac{G}{\omega} \end{bmatrix}    \begin{bmatrix}   g_y \\ -  g_x \end{bmatrix},   
 \quad \text{or equivalently,}    \quad
  \begin{bmatrix}   g_x \\ g_y \end{bmatrix}  = - \frac{1}{\mu}   \begin{bmatrix}  \frac{E}{\omega} & \frac{F}{\omega} 
  \\ \frac{F}{\omega} & \frac{G}{\omega} \end{bmatrix}    \begin{bmatrix}   f_y \\ -  f_x \end{bmatrix}.
\]
 Then, the graph $\Sigma$ is a minimal surface  in ${\mathbb{R}}^{4}$. Moreover, its generalized Gauss map (Definition \ref{def of generalized gauss map}) lies on
the hyperplane $z_{3} + i \mu z_{4} =0$ of the complex projective space ${\mathbb{CP}}^{3}$. 
 \end{theorem}

 The \emph{Lagrange potential} (Lemma \ref{Lagrange potential} and Remark \ref{Geometric meaning of Lagrange potential}) on minimal graphs in ${\mathbb{R}}^{3}$ plays a critical role in the Jenkins-Serrin construction \cite[Section 3]{Jenkins Serrin 1966} of minimal graphs with infinite boundary values. As stated in \cite[Section 2]{Bers 1951}, when the height of a minimal graph in ${\mathbb{R}}^{3}$ is interpreted as the potential of a flow of the Chaplygin gas, our Lagrange potential becomes the 
 \emph{stream function}. We use the Lagrange potential to build two dimensional minimal graphs in ${\mathbb{R}}^{4}$ and three dimensional minimal graphs in ${\mathbb{R}}^{6}$. 

\begin{theorem}[\textbf{Two applications of Lagrange potential of height on minimal surfaces}] \label{main two}
Let ${\Sigma}_{0}$ be the minimal graph of the function $\mathbf{p} : \Omega \to \mathbb{R}$ defined on a domain  $\Omega \subset {\mathbb{R}}^{2}:$
\[
{\Sigma}_{0} =  \left\{   \left(x, y, \mathbf{p}(x, y) \right) \in {\mathbb{R}}^{3} \, \vert \, (x,y) \in \Omega \right\}, 
\]
where the height function $p(x,y)$ solves the minimal surface equation
\[
0 =   \frac{\partial}{\partial x} \left(  \frac{ {\mathbf{p}}_x }{\sqrt{1+{{\mathbf{p}}_{x}}^{2}+{{\mathbf{p}}_{y}}^{2} }}  \right)+  
 \frac{\partial}{\partial y} \left(  \frac{ {\mathbf{p}}_y }{\sqrt{1+{{\mathbf{p}}_{x}}^{2}+{{\mathbf{p}}_{y}}^{2} }}  \right), \quad \left(x, y \right) \in \Omega.
\]
Let ${\mathbf{q}} : \Omega \to \mathbb{R}$ denote the Lagrange potential of ${\mathbf{p}} : \Omega \to \mathbb{R}$, which solves the system
\begin{equation}  \label{Lag conjugate 2}
 \left({\mathbf{q}}_{x}, {\mathbf{q}}_{y} \right)= \left( - \frac{ {\mathbf{p}}_y }{\sqrt{1+{{\mathbf{p}}_{x}}^{2}+{{\mathbf{p}}_{y}}^{2} }}, \frac{ {\mathbf{p}}_x }{\sqrt{1+{{\mathbf{p}}_{x}}^{2}+{{\mathbf{p}}_{y}}^{2} }}    \right). 
\end{equation}  
\begin{enumerate}
\item[\textbf{(a)}] For any constant $\lambda \in \mathbb{R}-\{0\}$, we associate the two dimensional graph  in ${\mathbb{R}}^{4}$
 \begin{equation} 
 {\Sigma}_{\lambda} =\left\{\,  \left(x, y,  \left( \cosh \lambda \right) \, {\mathbf{p}}(x, y),  \left( \sinh \lambda \right) \,  {\mathbf{q}}(x, y) \right) \in {\mathbb{R}}^{4} \, \vert \, (x,y) \in \Omega \, \right\}.
\end{equation}  
Then, the pair $\left(  f(x,y), g(x,y)  \right) = \left(  \left( \cosh \lambda \right) {\mathbf{p}}(x, y),   \left( \sinh \lambda \right) {\mathbf{q}}(x, y) \right)$ obeys
the Osserman system with the coefficient $\mu = \coth \lambda$. In particular, the graph ${\Sigma}_{\lambda}$ becomes  
a minimal surface in ${\mathbb{R}}^{4}$. Furthermore, we obtain the invariance of the conformally changed induced metric 
$\frac{1}{\sqrt{\det{ \left( {\mathbf{g}}_{{\Sigma}_{\lambda}}  \right) }}}  {\mathbf{g}}_{{\Sigma}_{\lambda}} = \frac{1}{\sqrt{\det{ \left( {\mathbf{g}}_{{\Sigma}_{0}}  \right) }}}  {\mathbf{g}}_{{\Sigma}_{0}}$. 
\item[\textbf{(b)}] For any constant $\lambda \in \mathbb{R}-\{0\}$, we associate the three dimensional graph  in ${\mathbb{R}}^{6}$
\[
 \left\{\, \left(x,\, y, \, z,\, {\mathbf{p}}_{x} - \lambda z \frac{ {\mathbf{p}}_{y} }{\sqrt{1+{{\mathbf{p}}_{x}}^{2}+{{\mathbf{p}}_{y}}^{2}}}, \,
 {\mathbf{p}}_{y} +\lambda z  \frac{ {\mathbf{p}}_{x} }{\sqrt{1+{{\mathbf{p}}_{x}}^{2}+{{\mathbf{p}}_{y}}^{2}}}, \, \lambda {\mathbf{q}}\, \right) \in {\mathbb{R}}^{6} \; \vert \; (x,y) \in \Omega, \, z \in \mathbb{R} \right\}.
\]
Then, it is special Lagrangian in ${\mathbb{C}}^{3}$, so homologically volume minimizing in ${\mathbb{R}}^{6}$.
\end{enumerate}
\end{theorem}

 We find the one parameter family of two dimensional minimal graphs in ${\mathbb{R}}^{4}$ defined over punctured plane, which realizes (part of) the Lagrangian catenoid in ${\mathbb{C}}^{2}$ and Fraser-Schoen band in ${\mathbb{R}}^{4}$ (Example \ref{punctured example}). We have the two parameter family of minimal graphs in ${\mathbb{R}}^{4}$ connecting catenoids, helicoids, planes in ${\mathbb{R}}^{3}$, and complex logarithmic graph in ${\mathbb{C}}^{2}$ (Example \ref{catenoid helicoid complex logarithmic graph}). 
 Applying the item \textbf{(a)} to helicoid and catenoid in ${\mathbb{R}}^{3}$, we obtain minimal surfaces in ${\mathbb{R}}^{4}$ foliated by hyperbolas or lines (Example \ref{Lee hyperbola}) and  the Osserman-Hoffman ellipse-foliated minimal annuli in ${\mathbb{R}}^{4}$ with total curvature $-4\pi$ (Example \ref{HO ellipse}) respectively. Applying  the item \textbf{(b)} to Scherk's graph ${\mathbb{R}}^{3}$, we build the doubly periodic special Lagrangian graphs in ${\mathbb{C}}^{3}$ (Example \ref{two periodic SL3}).

\section{Minimal surface system in ${\mathbb{R}}^{4}$ and Cauchy--Riemann equations}

 Our ambient space is Euclidean space ${\mathbb{R}}^{4}$ equipped with the metric ${dx_{1}}^{2} +\cdots + {dx_{4}}^{2}$.

\begin{theorem}[\textbf{Two dimensional minimal graphs in ${\mathbb{R}}^{4}$}]  \label{MSS}
Let ${\Sigma}$ denote the graph in ${\mathbb{R}}^{4}$ 
\[
{\Sigma}^{2}=\left\{\, \mathbf{\Phi}(x,y) = \left(x, y, f(x, y), g(x,y) \right) \in {\mathbb{R}}^{4} \, \vert \, (x,y) \in \Omega \, \right\}.
\]
 The induced metric ${\mathbf{g}}_{{\Sigma}}$ on the surface $\Sigma$ reads
\begin{equation*} \label{FirstFF 1}
{\mathbf{g}}_{{\Sigma}} = E dx^{2} + 2F dx dy  + G dy^{2}, 
\end{equation*}
where the coefficients of the first fundamental form are determined by 
\begin{equation} \label{FirstFF 2}
E = {\mathbf{\Phi}}_{x} \cdot {\mathbf{\Phi}}_{x} = 1+ {f_{x}}^{2}+{g_{x}}^{2}, \;
F = {\mathbf{\Phi}}_{x} \cdot {\mathbf{\Phi}}_{y} =   f_{x} f_{y}  +  g_{x} g_{y}, \;
G = {\mathbf{\Phi}}_{y} \cdot {\mathbf{\Phi}}_{y} = 1+ {f_{y}}^{2}+{g_{y}}^{2}.
\end{equation}
The area element is $ dA_{\Sigma} = \omega \, dx \, dy$ wtih $\omega= \sqrt{EG -F^2}$. 
We introduce the minimal surface operator ${\mathcal{L}}_{\Sigma}$ and Laplace-Beltrami operator ${\triangle}_{\Sigma}$ acting on functions on $\Omega:$
\begin{equation} \label{MSO}
{\mathcal{L}}_{{}_{\Sigma}} := G  \frac{{\partial}^{2} }{ \partial x^{2}} - 2F  \frac{{\partial}^{2}}{ \partial x \partial y} + E \frac{{\partial}^{2}}{ \partial y^{2}},
\end{equation}
\begin{equation} \label{LBO}
{\triangle}_{{}_{\Sigma}} :={\triangle}_{{\mathbf{g}}_{{\Sigma}} }  = \frac{1}{\omega} \left[ \; \frac{\partial}{\partial x}  \left( \frac{G}{\omega} \frac{\partial}{\partial x} \; - 
 \frac{F}{\omega} \frac{\partial}{\partial y} \; \right) + \frac{\partial}{\partial y}  \left( - \frac{F}{\omega} \frac{\partial}{\partial x} \; +
 \frac{E}{\omega} \frac{\partial}{\partial y}  \; \right)  \; \right]. 
\end{equation}
Then, the following three conditions are equivalent. 
\begin{enumerate}
\item[\textbf{(a)}] Two height functions $f(x,y)$ and $g(x,y)$ are harmonic on the graph $\Sigma$$:$
\[
{\triangle}_{{}_{\Sigma}} f = 0 \quad \text{and}  \quad {\triangle}_{{}_{\Sigma}} g = 0.
\]
\item[\textbf{(b)}] The graph $\Sigma$ is minimal in ${\mathbb{R}}^{4}$.
\item[\textbf{(c)}] Two height functions $f(x,y)$ and $g(x,y)$ solve the minimal surface system$:$ 
\begin{equation} \label{MSS quote}
{\mathcal{L}}_{{}_{\Sigma}}   f = 0 \quad \text{and}  \quad {\mathcal{L}}_{{}_{\Sigma}} g = 0.
\end{equation}
\end{enumerate}
\end{theorem}

\begin{proof}
The equivalence of \textbf{(a)} and \textbf{(b)} follows from \cite[Equation (3.14) in Section 2]{Osserman 1986}, which indicates that the Euler-Lagrange system for the area functional of the graph is
\[
\frac{\partial}{\partial x}  \left( \frac{G}{\omega} \begin{bmatrix}   f_x \\ g_x \end{bmatrix}   \; - 
 \frac{F}{\omega}  \begin{bmatrix}   f_y \\ g_y \end{bmatrix}  \; \right) + \frac{\partial}{\partial y}  \left( - \frac{F}{\omega}  \begin{bmatrix}   f_x \\ g_x \end{bmatrix}  \; +
 \frac{E}{\omega}  \begin{bmatrix}   f_y \\ g_y \end{bmatrix}   \; \right) = \begin{bmatrix}  0 \\ 0 \end{bmatrix} \;\;
\Leftrightarrow \;\; {\triangle}_{{}_{\Sigma}} \begin{bmatrix}   f_x \\ g_x \end{bmatrix} = \begin{bmatrix}  0 \\ 0 \end{bmatrix}.
\]
There are several ways to establish the equivalence of \textbf{(b)} and \textbf{(c)}:  \cite[Section 2, p. 16-17]{Osserman 1986}, 
\cite[Section 2]{Lawson Osserman 1977}, \cite[Section 1.2]{Becker Kahn 2014}  (for arbitrary codimension), 
\cite[Appendix: The minimal surface system]{Micallef White1984}, and \cite[Example 1]{Eells 1979} (for more general ambient spaces).
 Here, we adopt the argument in the proof of \cite[Theorem 2.2]{Osserman 1969}. 
We use the formula (\ref{LBO}) and introduce
\begin{equation} \label{PQ pair 1}
\left( \mathcal{P},  \mathcal{Q} \right)  = \left( \,  \frac{\partial}{\partial x}  \left( \frac{G}{\omega} \right) - \frac{\partial}{\partial y}  \left( \frac{F}{\omega} \right) , \, 
\frac{\partial}{\partial y}  \left( \frac{E}{\omega} \right) - \frac{\partial}{\partial x}  \left( \frac{F}{\omega} \right) \, \right) . 
\end{equation}
to obtain the identity for the mean curvature vector $\mathbf{H}(x,y) :={\triangle}_{{}_{\Sigma}}  \mathbf{\Phi}(x,y)$ explicitly$:$
\begin{equation} \label{MCV general}
\mathbf{H}  =  
{\triangle}_{{}_{\Sigma}}  \begin{bmatrix}  x  \\  y  \\   f(x,y) \\   g(x,y) \end{bmatrix}
= \frac{1}{\omega}  \begin{bmatrix}   \mathcal{P}  \\   \mathcal{Q}  \\ 
  \mathcal{P}  f_{x} +  \mathcal{Q}  f_{y} +  \frac{ 1 }{{\omega}} {\mathcal{L}}_{{}_{\Sigma}} f  \\ 
 \mathcal{P}   g_{x} + \mathcal{Q}  g_{y} +\frac{ 1 }{{\omega}} {\mathcal{L}}_{{}_{\Sigma}} g  \end{bmatrix}
  = \frac{\mathcal{P}}{\omega}  {\mathbf{\Phi}}_{x} +  \frac{\mathcal{Q}}{\omega}  {\mathbf{\Phi}}_{y}  + \frac{1}{{\omega}^{2}}  \begin{bmatrix}  0  \\  0  \\ {\mathcal{L}}_{{}_{\Sigma}} f  \\  {\mathcal{L}}_{{}_{\Sigma}} g  \end{bmatrix}.
\end{equation}
First, assume \textbf{(b)}. Since the mean curvature vector $\mathbf{H}$ vanishes on the minimal surface, we find that (\ref{MCV general}) guarantees that four quantities 
$\mathcal{P}$, $\mathcal{Q}$, ${\mathcal{L}}_{{}_{\Sigma}}   f$, ${\mathcal{L}}_{{}_{\Sigma}}   g$ vanish. So, \textbf{(c)} holds. 
Second, assume  \textbf{(c)}. 
Since ${\mathcal{L}}_{{}_{\Sigma}}   f = 0$ and  ${\mathcal{L}}_{{}_{\Sigma}} g = 0$, we find that, by (\ref{MCV general}), the mean curvature vector $\mathbf{H}$ is equal to the tangent vector $\frac{\mathcal{P}}{\omega}  {\mathbf{\Phi}}_{x} +  \frac{\mathcal{Q}}{\omega}  {\mathbf{\Phi}}_{y}$. As the mean curvature vector $\mathbf{H}$ is normal to the 
graph $\Sigma$, we conclude that $\mathbf{H}$ vanishes. So, \textbf{(b)} holds. 
 \end{proof}
 
 \begin{remark}[\textbf{Minimal surface operator ${\mathcal{L}}_{\Sigma}$ and Laplace-Beltrami operator ${\triangle}_{\Sigma}$}] \label{two equiv operators}
When we assume that the two dimensional minimal graph $\Sigma$ is minimal in ${\mathbb{R}}^{4}$, we obtain 
\[
  {\triangle}_{\Sigma} = \frac{G}{{\omega}^{2}}  \frac{{\partial}^{2} }{ \partial x^{2}} - 2\frac{F}{{\omega}^{2}}  \frac{{\partial}^{2}}{ \partial x \partial y} + \frac{E}{{\omega}^{2}} \frac{{\partial}^{2}}{ \partial y^{2}}, \quad  \text{or briefly,} \quad   {\triangle}_{\Sigma} = \frac{1}{{\omega}^{2}} {\mathcal{L}}_{\Sigma}.
\] 
The minimality of the graph $\Sigma$ implies the two interesting identities
\begin{equation} \label{classical equalities}
\frac{\partial}{\partial y} \left( \frac{F}{\omega} \right) = \frac{\partial}{\partial x} \left( \frac{G}{\omega} \right) \quad \text{and} \quad 
\frac{\partial}{\partial y} \left( \frac{E}{\omega} \right) = \frac{\partial}{\partial x} \left( \frac{F}{\omega} \right),
\end{equation} 
which imply 
\[
 {\triangle}_{\Sigma}  = \frac{1}{{\omega}^{2}} {\mathcal{L}}_{\Sigma} + \left[\frac{\partial}{\partial x} \left( \frac{G}{\omega} \right) - \frac{\partial}{\partial y} \left( \frac{F}{\omega} \right) \right] \frac{\partial}{\partial x} +  \left[ \frac{\partial}{\partial y} \left( \frac{E}{\omega} \right) -  \frac{\partial}{\partial x} \left( \frac{F}{\omega} \right) \right] \frac{\partial}{\partial y} =  \frac{1}{{\omega}^{2}} {\mathcal{L}}_{\Sigma}.  
\]
A geometric implication of  (\ref{classical equalities}) on minimal graphs can be found in Rado's book \cite[p. 108]{Rado 1933}.  
A variational proof of (\ref{classical equalities}) is given in Osserman's book \cite[Chapter 3]{Osserman 1986}.
Another interpretation of (\ref{classical equalities}) (via conjugate minimal surface) is illustrated in 
Remark \ref{Geometric meaning of Lagrange potential}. 
\end{remark}

By Bers' Theorem \cite{Bers 1951, Finn 1953, Osserman 1973}, an isolated singularity of a single valued solution to the minimal surface equation  
is removable. However, in higher codimensions, the minimal surface system can have solutions with isolated non-removable singularities.

 \begin{example}[\textbf{Two dimensional minimal graphs in ${\mathbb{R}}^{4}$ over the punctured plane}]  \label{punctured example}
 Let $N$ be a positive integer. We shall modify the entire holomorphic graph in ${\mathbb{C}}^{2}={\mathbb{R}}^{4}$
 \[
\Sigma_{\textrm{entire}} = \left\{\; \left(x, y, \textbf{Re}\left[ {\left( x+ iy\right)}^{N} \right], \textbf{Im}\left[ {\left( x+ iy\right)}^{N} \right] \right)
  \in {\mathbb{R}}^{4} \; \vert \; (x,y) \in {\mathbb{R}}^{2} \; \right\}.
 \]
 Let ${\textbf{T}}_{N}\left(\zeta\right)$ denote the Chebyshev polynomial of the first kind and degree $N$. The Chebyshev polynomials 
 are determined by the recurrence relation
 \[
{\textbf{T}}_{0}\left(\zeta\right)=1, \;\; {\textbf{T}}_{1}\left(\zeta\right)=\zeta, \;\; {\textbf{T}}_{2}\left(\zeta\right)=2{\zeta}^{2}-1, \;\;
 {\textbf{T}}_{k+2}\left(\zeta\right) =2 \zeta {\textbf{T}}_{k+1}\left(\zeta\right) - {\textbf{T}}_{k}\left(\zeta\right) \, \text{for} \; k \geq 0.
 \]
The Chebyshev polynomial ${\textbf{T}}_{N}\left(\zeta\right)$ is characterized by the properties   
 \[
 {\textbf{T}}_{N}\left(\zeta\right) = 
 \begin{cases}
   \cosh \left( N\, \text{arcosh}  \, \zeta \right)  \quad \text{for} \;\, \zeta \in \left[1, \infty\right], \\
   \cos \left( N\,  \arccos \, \zeta \right)  \quad \text{for} \;\, \zeta \in \left[-1, 1\right].
 \end{cases}
 \]
 We construct the following minimal graph over the punctured plane
  \[
\Sigma_{N} = \left\{\; \left(x, y, {\Psi}_{N}(\rho)  \, \textbf{Re}\left[ {\left( x+ iy\right)}^{N} \right],  {\Psi}_{N}( \rho ) \, \textbf{Im}\left[ {\left( x+ iy\right)}^{N} \right] \right)
  \in {\mathbb{R}}^{4} \; \vert \; (x,y) \in {\mathbb{R}}^{2}-\{\left(0, 0\right)\} \; \right\},
 \]
 where we introduce the radially symmetric weight function ${\Psi}_{N}(\rho)$: 
 \[
  \rho=\sqrt{x^{2}+y^{2}} \quad \text{and} \quad {\Psi}_{N}(\rho)= \frac{1}{N} \cdot \frac{  {\textbf{T}}_{N}\left( \sqrt{1+ {\rho}^{2}} \right)      }{{\rho}^{N}} = 
  \frac{1}{N} \cdot \frac{  {\textbf{T}}_{N}\left( \sqrt{1+ x^2 +y^2} \right)      }{{\left( x^2 +y^2 \right)}^{\frac{N}{2}}}. 
 \] 
 \begin{enumerate}
  \item[\textbf{(a)}] The new coordinates $\left(t, \theta \right) \in \left(0, \infty \right) \times \left[0, 2\pi \right]$ with 
  $\left(x, y \right)=\left(\sinh t \cos \theta, \sinh t \sin \theta \right)$ induce the conformal harmonic 
  patch which admits the extension $\mathbf{X}_{N}(t, \theta)$: 
 \[
(t, \theta) \in {\mathbb{R}} \times \left[0, 2\pi \right]  \to \left( \sinh t \cos \theta, \sinh t \sin \theta, \frac{1}{N}  \cosh \left( N t \right) \cos \left( N \theta \right),  \frac{1}{N} 
  \cosh \left( N t \right) \sin \left( N \theta \right)            \right) 
 \]
 of the minimal surface $\widetilde{\Sigma_{N}}$ in ${\mathbb{R}}^{4}$ with the induced conformal metric
 \[
  {\mathbf{g}}_{\widetilde{\Sigma_{N}}} 
  = \left[ {\cosh}^{2}t + {\sinh}^{2}  \left( N t \right)  \right] \left( dt^2 + {d\theta}^{2} \right) 
  = \left[ {\sinh}^{2}t + {\cosh}^{2}  \left( N t \right)  \right] \left( dt^2 + {d\theta}^{2} \right). 
 \]
 \item[\textbf{(b)}] When $N=1$, we have the minimal gradient graph ${\Sigma}_{1}$ explicitly given by  
 \[
\left\{\; \left( \; x, \, y, \, x \sqrt{1+\frac{1}{x^2 +y^2 \,}}, \,y \sqrt{1+\frac{1}{x^2 +y^2 \,}} \; \right)
  \in {\mathbb{R}}^{4} \;\; \vert \;\; (x,y) \in {\mathbb{R}}^{2}-\{\left(0, 0\right)\} \; \right\},
 \]
As Osserman \cite{Osserman 1973} observed, we see that the pair of height functions 
\[
\left(f(x,y), g(x,y) \right)=  \sqrt{1+x^2 +y^2}\left(\frac{x}{\sqrt{x^2 +y^2}}, \frac{y}{\sqrt{x^2 +y^2}}\right)
\]
has an isolated singularity at the origin. Neither height functions $f(x,y)$ nor $g(x,y)$ tends to a limit at the origin. 
 The \textbf{Lagrangian catenoid} (\cite[Theorem 3.5]{Harvey Lawson 1982} and \cite[Example 6.4]{Lee 2017}) realizes 
 the surface given by the patch $\mathbf{X}_{1}(t, \theta)$.
 \item[\textbf{(c)}] When $N=2$, the minimal graph ${\Sigma}_{2}$ becomes a rational variety  
  \[
 \left\{\; \left( \; x, \, y, \, \left( x^2 -y^2 \right) \left( 1+ \frac{1}{2 \left( x^2 +y^2 \, \right)} \right)  , \,\left( 2xy \right) \left( 1+ \frac{1}{2 \left( x^2 +y^2 \,\right)} \right) \; \right)
  \in {\mathbb{R}}^{4} \;\; \vert \;\; (x,y) \in {\mathbb{R}}^{2}-\{\left(0, 0\right)\} \; \right\}.
 \]
It is the graph of the non-holomorphic function $ \zeta =x+iy \in \mathbb{C}-\{0\} \mapsto {\zeta}^{2} + \frac{1}{2}  {\left(\, \overline{\zeta} \,\right)}^{-1}   \zeta$.
  The \textbf{Fraser-Schoen band} \cite[Section 7]{Fraser Schoen 2016} is the minimal embedding  $\mathbf{X}_{2}(t, \theta)$ of a M\"{o}bius band 
 into ${\mathbb{R}}^{4}$. This induces, after a suitable rescaling, the free boundary minimal surface in the four dimensional unit ball
 whose  coordinate functions become first Steklov eigenfunctions \cite[Proposition 7.1]{Fraser Schoen 2016}.
 \end{enumerate}
 \end{example}
 
\begin{remark}
Osserman \cite{Osserman 1973} extended Bers' Theorem to the minimal 
surface system in arbitrary codimensions, assuming the continuity of all but one of the height functions. 
\end{remark}

\begin{example}[\textbf{Two parameter family of minimal graphs in ${\mathbb{R}}^{4}$ connecting catenoids, helicoids in ${\mathbb{R}}^{3}$, and complex logarithmic graph in ${\mathbb{C}}^{2}$}]  \label{catenoid helicoid complex logarithmic graph}
 Given a pair $(\alpha, \beta) \in {\mathbb{R}}^{2}$ of constants, we associate the following two dimensional graph ${\Sigma}_{\left(\alpha, \beta\right)}$ in ${\mathbb{R}}^{4}$
\[
{\Sigma}_{\left(\alpha, \beta\right)} =\left\{\left(x, y, \alpha \ln \left( \frac{\sqrt{x^2 +y^2} +\sqrt{x^2 +y^2 +{\beta}^{2} - {\alpha}^{2}}  }{2} \right), 
\beta \arctan\left( \frac{y}{x} \right) \right) \in {\mathbb{R}}^{4} \; \vert \; (x,y) \in \Omega \, \right\}.
\]
The pair $\left( f(x,y), g(x,y) \right) = \left( \alpha \ln \left( \frac{r +\sqrt{r^2 +{\beta}^{2} - {\alpha}^{2}}  }{2} \right), 
\beta \arctan\left( \frac{y}{x} \right)   \right)$ with $r=\sqrt{x^2 +y^2}$ solves the minimal surface system (\ref{MSS quote}) in 
Theorem \ref{MSS quote}. We shall distinguish the three cases:
\begin{enumerate}
\item[\textbf{(a)}] When $(\alpha, \beta)=\left(\cosh \lambda, \sinh \lambda \right)$ for some constant $\lambda \in \mathbb{R}$, recall the identity
\[
\textrm{arcosh} \, r=\ln \left(r+\sqrt{r^2 -1} \right), \quad r \geq 1.
\]
We see that, up to translations, the graph ${\Sigma}_{\left(\alpha, \beta\right)}$ is congruent to  
\[
{\Sigma}^{+}_{\lambda} = \left\{\left(x, y, \left( \cosh \lambda \right) \, \textrm{arcosh} \left(\sqrt{x^2 +y^2}\right), \left( \sinh \lambda \right)\, \arctan\left( \frac{y}{x} \right) \right) \in {\mathbb{R}}^{4} \; \vert \; x^2 +y^2 \geq 1, \, x \neq 0 \right\}.
\]
When  $\lambda=0$, the surface ${\Sigma}^{+}_{0}$ recovers the \textbf{catenoid} foliated by circles. 
\item[\textbf{(b)}] When $(\alpha, \beta)=\left(\lambda, \lambda \right)$ for some  $\lambda \in \mathbb{R}-\{0\}$, the non-planar minimal surface ${\Sigma}_{\left(\alpha, \beta\right)}$ in ${\mathbb{R}}^{4}$ can be identified as the \textbf{complex logarithmic graph} in ${\mathbb{C}}^{2}$:
\[
{\Sigma}^{0}_{\lambda} =  \left\{  \left(\zeta,  \lambda \log \zeta\right)  \in {\mathbb{C}}^{2} \; \vert \; \zeta \in \mathbb{C}-\{0\} \right\}.
\]
\item[\textbf{(c)}] When $(\alpha, \beta)=\left(\sinh \lambda, \cosh \lambda \right)$ for some constant $\lambda \in \mathbb{R}$, recall the identity
\[
\textrm{arsinh} \, r=\ln \left(r+\sqrt{r^2 +1} \right), \quad r \in \mathbb{R}.
\] 
We find that, up to translations, ${\Sigma}_{\left(\alpha, \beta\right)}$ is congruent to  
\[
{\Sigma}^{-}_{\lambda} =  \left\{\left(x, y, \left( \sinh \lambda \right)\, \textrm{arsinh} \left(\sqrt{x^2 +y^2}\right), \left(\cosh \lambda\right)\, \arctan\left( \frac{y}{x} \right) \right) \in {\mathbb{R}}^{4} \; \vert \; (x,y) \in 
 \left( {\mathbb{R}}-\{0\} \right) \times {\mathbb{R}} \, \right\}.
\]
In the case when  $\lambda=0$,  the surface ${\Sigma}^{-}_{0}$ recovers the \textbf{helicoid} foliated by lines. 
\end{enumerate}
\end{example}

\begin{proposition}[\textbf{Cauchy--Riemann equations on the two dimensional minimal graph}]  \label{CR on graph lemma}
Assume that the graph $\Sigma =\left\{   \left(x, y, f(x, y), g(x,y) \right) \in {\mathbb{R}}^{4} \, \vert \, (x,y) \in \Omega \right\}$ is 
 minimal in ${\mathbb{R}}^{4}$. Then, the $\mathbb{C}$-valued function ${\mathcal{A}}(x, y)+i {\mathcal{B}}(x,y)$ is holomorphic 
 on $\Sigma$ when we have  
\begin{equation} \label{CR on graph}
  \begin{bmatrix} {\mathcal{A}}_x \\ {\mathcal{A}}_y \end{bmatrix}  =    \begin{bmatrix}  \frac{E}{\omega} & \frac{F}{\omega} 
  \\ \frac{F}{\omega} & \frac{G}{\omega} \end{bmatrix}    \begin{bmatrix}   {\mathcal{B}}_y \\ -  {\mathcal{B}}_x \end{bmatrix},
 \quad \text{or equivalently,} \quad 
   \begin{bmatrix} {\mathcal{B}}_x \\ {\mathcal{B}}_y \end{bmatrix}  =  -  \begin{bmatrix}  \frac{E}{\omega} & \frac{F}{\omega} 
  \\ \frac{F}{\omega} & \frac{G}{\omega} \end{bmatrix}    \begin{bmatrix}   {\mathcal{A}}_y \\ - {\mathcal{A}}_x \end{bmatrix}.
\end{equation}  
\end{proposition}

\begin{proof} 
We use the special construction of the local conformal coordinates on minimal graphs to deduce the Cauchy--Riemann 
equations (\ref{CR on graph}) on $\Sigma$. Due to the identities (\ref{classical equalities}):
\[
\frac{\partial}{\partial y} \left( \frac{F}{\omega} \right) = \frac{\partial}{\partial x} \left( \frac{G}{\omega} \right) \quad \text{and} \quad 
\frac{\partial}{\partial y} \left( \frac{E}{\omega} \right) = \frac{\partial}{\partial x} \left( \frac{F}{\omega} \right),
\]
we can find the potential functions $\mathbf{M}(x,y)$ and $\mathbf{N}(x,y)$ so that
\[
 \left( \mathbf{M}_{x}, \mathbf{M}_{y} \right) = \left( \frac{F}{\omega},  \frac{G}{\omega} \right) \quad \text{and} \quad 
 \left( \mathbf{N}_{x}, \mathbf{N}_{y} \right) = \left( \frac{E}{\omega},  \frac{F}{\omega} \right), 
\] 
at least in a sufficiently small neighborhood of any point in $\Omega$. 
As in \cite[Lemma 4.4]{Osserman 1986}, 
\[
  \left(x, y \right) \to \Xi(x,y) = \left( {\xi}_{1},  {\xi}_{2} \right) =\left( x+ \mathbf{M}(x, y), y+\mathbf{N}(x,y) \right)
\]
becomes the local diffeomorphism, and gives the desired local conformal coordininates $\left( {\xi}_{1},  {\xi}_{2} \right)$ on 
the minimal graph $\Sigma$ in ${\mathbb{R}}^{4}$ equipped with the induced conformal metric 
 \[
 {\mathbf{g}}_{{\Sigma}} =  \frac{\omega}{2+\frac{E}{\omega} +\frac{G}{\omega}}   \left( d{{\xi}_{1}}^{2} +  d{{\xi}_{2}}^{2} \right).
 \] 
The $\mathbb{C}$-valued function ${\mathcal{A}}(x, y)+i {\mathcal{B}}(x,y)$ is holomorphic on $\Sigma$ with respect 
 to the conformal coordinates $\left( {\xi}_{1},  {\xi}_{2} \right)= \Xi(x,y)$ if and only if we have the Cauchy--Riemann equations 
 \[
 \begin{bmatrix}    \frac{\partial}{\partial {\xi}_{1}}  \left( {\mathcal{A}} \circ {\Xi}^{-1} \right)  \\    
 \frac{\partial}{\partial {\xi}_{2}}  \left( {\mathcal{A}} \circ {\Xi}^{-1} \right)  \end{bmatrix}  = 
 \begin{bmatrix}  \;  \frac{\partial}{\partial {\xi}_{2}}  \left( {\mathcal{B}} \circ {\Xi}^{-1} \right)  \\    
- \frac{\partial}{\partial {\xi}_{1}}  \left( {\mathcal{B}} \circ {\Xi}^{-1} \right)  \end{bmatrix},
\]
which could be transformed to the desired system (\ref{CR on graph}) via the chain rule.
\end{proof}

\begin{remark}[\textbf{Beltrami equations, \cite{Beyerstedt 1991}}] The system 
$ \begin{bmatrix} {\mathcal{B}}_x \\ {\mathcal{B}}_y \end{bmatrix}  =  -  \begin{bmatrix}  \frac{E}{\omega} & \frac{F}{\omega} 
  \\ \frac{F}{\omega} & \frac{G}{\omega} \end{bmatrix}    \begin{bmatrix}   {\mathcal{A}}_y \\ - {\mathcal{A}}_x \end{bmatrix}$ is 
  the Beltrami equations associated to the metric ${\mathbf{g}}_{{\Sigma}} = E dx^{2} + 2F dx dy  + G dy^{2}$ 
  with $\omega=\sqrt{EG-F^{2}}$.
\end{remark}

\section{Generalized Gauss map and Osserman system of the first order}

The main point of this section is to build our Osserman system of the first order, which generalizes the Cauchy--Riemann equations,
and to provide the proof of Theorem \ref{why osserman}, which illustrates the birth of the Osserman system. 
  As in \cite{Chern Osserman 1967, Hoffman Osserman 1980, Osserman 1969, Osserman 1986}, we introduce the generalized Gauss map of minimal surfaces in ${\mathbb{R}}^{4}$. Inside the complex projective space  ${\mathbb{C}}{\mathbb{P}}^{3}$, we prepare the complex hyperquadric
\[
{\mathcal{Q}_{2}}=\{ \, \left[ z_1 : z_2 : z_3 : z_4 \right] \in {\mathbb{CP}}^{3} \;  \vert \;  {z_1}^{2}+{z_2}^{2}+{z_3}^{2}+{z_4}^{2}=0 \, \}.
\]

\begin{definition}[\textbf{Generalized Gauss map of minimal surfaces in ${\mathbb{R}}^{4}$,  \cite[Section 2]{{Osserman 1986}}}] \label{def of generalized gauss map}
Let ${\Sigma}^{2}$ be a minimal surface in ${\mathbb{R}}^{4}$. Consider a conformal harmonic immersion $X:\Sigma \rightarrow {\mathbb{R}}^{4}$, $\xi \mapsto X(\xi)$. The generalized Gauss map of $\Sigma$ is the map
$\mathcal{G}:\Sigma \rightarrow {\mathcal{Q}}_{2} \subset {\mathbb{C}}{\mathbb{P}}^{3}$ defined by 
\[
   \mathcal{G}(\xi)=\left[ \;   \overline{ \frac{\partial X}{\partial {\xi}} } \; \right]
   = \left[ \frac{\partial X}{\partial {\xi}_{1}}  + i \frac{\partial X}{\partial {\xi}_{2}} \right] \in {\mathcal{Q}}_{2}.
\]
The conformality of the immersion $X$ guarantees that the generalized Gauss map is a well-defined ${\mathcal{Q}}_{2}$-valued function. 
The harmonicity of the immersion $X$ guarantees that the generalized Gauss map is anti-holomorphic.
\end{definition}

 \begin{lemma}[\textbf{Generalized Gauss map of two dimensional minimal graphs in ${\mathbb{R}}^{4}$}] \label{Gauss map for minimal graphs}
 Let 
 \[
 {\Sigma}^{2}=\left\{\,  \left(x, y, f(x, y), g(x,y) \right) \in {\mathbb{R}}^{4} \, \vert \, (x,y) \in \Omega \, \right\}
\]
 denote a minimal graph in ${\mathbb{R}}^{4}$ with the metric $E dx^{2} + 2F dx dy  + G dy^{2}$ and 
 $\omega= \sqrt{EG -F^2}$. Its generalized Gauss map $\mathcal{G}:\Omega \rightarrow {\mathcal{Q}}_{2} \subset {\mathbb{C}}{\mathbb{P}}^{3}$ can be
 explicitly given in terms of the coordinates $(x, y):$
\begin{eqnarray*} \label{graph gauss map}
  \mathcal{G}(x,y) &=&  \left[ z_1 : z_2 : z_3 : z_4 \right] \\
   &=& \left[  \frac{G}{\omega}: \;  i- \frac{F}{\omega}: \;  \frac{G}{\omega} f_x +  \left(  i - \frac{F}{\omega} \right)  f_y : \;  \frac{G}{\omega} g_x + \left(  i - \frac{F}{\omega} \right)  g_y \right]  \\
 &=& \left[  1-i\frac{F}{\omega}: \; i\frac{E}{\omega}: \; \left(  1 - i \frac{F}{\omega} \right) f_x + i \frac{E}{\omega}  f_y: \; \left(  1 - i \frac{F}{\omega} \right)  g_x + i \frac{E}{\omega} g_y  \right].
    \end{eqnarray*}
\end{lemma}

\begin{proof}
For the details of the deduction of Lemma \ref{Gauss map for minimal graphs}, we refer to \cite[Proposition 6]{Lee 2013}, which 
was inspired by the equality in \cite[Lemma, p. 290]{Osserman 1968}. 
\end{proof}
 
 \begin{definition}[\textbf{Degenerate minimal surfaces in ${\mathbb{R}}^{4}$,  \cite[Section 2, p. 122]{Osserman 1986}}] \label{degenerate def}
 We say that a minimal surface $\Sigma$ in ${\mathbb{R}}^{4}$ is \textbf{degenerate} if the image of its ${\mathcal{Q}}_{2}$-valued  generalized 
 Gauss map lies in a hyperplane of the complex projective space ${\mathbb{CP}}^{3}$. 
\end{definition}

\begin{remark}[\textbf{Gauss maps and representation of degenerate minimal surfaces in ${\mathbb{R}}^{4}$}]
For a geometric illustration of generalized Gauss map of degenerate minimal surfaces, we refer to 
\cite[Figure 1]{Chen Goes 1983}. As known in \cite[Theorem 4.7]{Hoffman Osserman 1980}, degenerate minimal 
surfaces  in ${\mathbb{R}}^{4}$ could be described by an explicit representation analogous to 
the Enneper-Weierstrass representation formula for minimal surfaces  in ${\mathbb{R}}^{3}$. 
\end{remark}

 \begin{remark}[\textbf{Degeneracy of entire two dimensional minimal graphs in arbitrary codimensions,  \cite[Chapter 5]{Osserman 1986}}]
Extending Bernstein's Theorem that the only entire minimal graphs in ${\mathbb{R}}^{3}$ are planes, 
Osserman established that the generalized Gauss map of entire two dimensional minimal graphs 
in ${\mathbb{R}}^{n+2 \geq 4}$ are degenerate. He also determined an explicit representation formula for 
entire two dimensional minimal graphs in ${\mathbb{R}}^{4}$ in terms of a single non-vanishing holomorphic function.
Though Osserman's complete non-planar minimal surfaces in ${\mathbb{R}}^{4}$ are entire graphs, they are not stable, except the holomorphic curves.
Micallef  \cite[Corollary 5.1]{Micallef 1984} established that if an entire two dimensional non-planar minimal graph in ${\mathbb{R}}^{4}$ is stable,
then it is a holomorphic curve in ${\mathbb{C}}^{2}$.
 \end{remark}

\begin{definition}[\textbf{Osserman system}]   \label{Osserman system DEF} 
Let ${\Sigma}$ be the graph in ${\mathbb{R}}^{4}$ of the pair $\left( f(x, y), g(x,y)  \right)$ of height functions 
defined on the domain $\Omega$:
\[
 \Sigma =\left\{\, \mathbf{\Phi}(x,y) = \left(x, y, f(x, y), g(x,y) \right) \in {\mathbb{R}}^{4} \, \vert \, (x,y) \in \Omega \, \right\}.
\]
We recall that the induced metric ${\mathbf{g}}_{{\Sigma}}$ and the area element on the surface $\Sigma$ reads
\[
{\mathbf{g}}_{{\Sigma}} = E dx^{2} + 2F dx dy  + G dy^{2}, \quad  dA_{\Sigma} = \omega \, dx \, dy, \quad \omega= \sqrt{EG -F^2},
\]
where the coefficients of the first fundamental form are determined by 
\[
E = {\mathbf{\Phi}}_{x} \cdot {\mathbf{\Phi}}_{x} = 1+ {f_{x}}^{2}+{g_{x}}^{2}, \;
F = {\mathbf{\Phi}}_{x} \cdot {\mathbf{\Phi}}_{y} =   f_{x} f_{y}  +  g_{x} g_{y}, \;
G = {\mathbf{\Phi}}_{y} \cdot {\mathbf{\Phi}}_{y} = 1+ {f_{y}}^{2}+{g_{y}}^{2}.
\]
Given a constant $\mu \in \mathbb{R}-\{0\}$, we introduce 
\begin{equation} \label{OSS definition 1}
  \begin{bmatrix}   f_x \\ f_y \end{bmatrix}  =  \mu   \begin{bmatrix}  \frac{E}{\omega} & \frac{F}{\omega} 
  \\ \frac{F}{\omega} & \frac{G}{\omega} \end{bmatrix}    \begin{bmatrix}   g_y \\ -  g_x \end{bmatrix},   
\end{equation}  
or equivalently,   
\begin{equation} \label{OSS definition 2}
  \begin{bmatrix}   g_x \\ g_y \end{bmatrix}  = - \frac{1}{\mu}   \begin{bmatrix}  \frac{E}{\omega} & \frac{F}{\omega} 
  \\ \frac{F}{\omega} & \frac{G}{\omega} \end{bmatrix}    \begin{bmatrix}   f_y \\ -  f_x \end{bmatrix},
\end{equation}  
which will be called the Osserman system with the coefficient $\mu \in \mathbb{R}-\{0\}$.
\end{definition}
 
 \begin{remark}
The definition $ \omega= \sqrt{EG -F^2}$ yields the formula
\[
{ \begin{bmatrix}  \frac{E}{\omega} & \frac{F}{\omega} 
  \\ \frac{F}{\omega} & \frac{G}{\omega} \end{bmatrix} } ^{-1} = \begin{bmatrix}  \frac{G}{\omega} & - \frac{F}{\omega} 
  \\ - \frac{F}{\omega} & \frac{E}{\omega} \end{bmatrix}.
  \] 
One can use this to check that the two systems (\ref{OSS definition 1}) and (\ref{OSS definition 2}) are equivalent to each other.
 \end{remark}
 
\begin{theorem}[\textbf{Minimality and degeneracy of Osserman graphs in ${\mathbb{R}}^{4}$}] \label{why osserman}
When the pair $\left( f(x,y), g(x,y) \right)$ satisfies the Osserman system (\ref{OSS definition 1}) with the coefficient  $\mu \in \mathbb{R}-\{0\}$, 
the graph $\Sigma =\left\{\,  \left(x, y, f(x, y), g(x,y) \right) \in {\mathbb{R}}^{4} \, \vert \, (x,y) \in \Omega \, \right\}$ is minimal in ${\mathbb{R}}^{4}$. 
Moreover, its generalized Gauss map lies on the hyperplane $z_{3} + i \mu z_{4} =0$ of the complex projective space ${\mathbb{CP}}^{3}$.
\end{theorem}
 
\begin{proof} 
To show the minimality of the graph
 $\Sigma =\left\{\,  \left(x, y, f(x, y), g(x,y) \right) \in {\mathbb{R}}^{4} \, \vert \, (x,y) \in \Omega \, \right\}$,
we employ Theorem \ref{MSS}.
Indeed, we use the equalities (\ref{OSS definition 2}) to obtain
\begin{eqnarray*} 
{\triangle}_{{}_{\Sigma}}  f  &=&    \frac{1}{\omega} \left[ \; \frac{\partial}{\partial x}  \left( \frac{G}{\omega} f_x \; - 
 \frac{F}{\omega} f_y\; \right) + \frac{\partial}{\partial y}  \left( - \frac{F}{\omega} f_x \; +
 \frac{E}{\omega} f_y  \; \right)  \; \right]   \\
 &=&  \frac{1}{\omega} \left[ \; \frac{\partial}{\partial x}  \left(  \mu g_y \right) + \frac{\partial}{\partial y}  \left(  - \mu g_x  \; \right)  \; \right]  \\
 &=& 0,
 \end{eqnarray*}
 and use the equalities in  (\ref{OSS definition 1})  to obtain
 \begin{eqnarray*} 
{\triangle}_{{}_{\Sigma}}  g  &=&    \frac{1}{\omega} \left[ \; \frac{\partial}{\partial x}  \left( \frac{G}{\omega} g_x \; - 
 \frac{F}{\omega} g_y\; \right) + \frac{\partial}{\partial y}  \left( - \frac{F}{\omega} g_x \; +
 \frac{E}{\omega} g_y  \; \right)  \; \right]   \\
 &=&  \frac{1}{\omega} \left[ \; \frac{\partial}{\partial x}  \left(  -\frac{1}{\mu} f_y \right) + \frac{\partial}{\partial y}  \left(  \frac{1}{\mu} f_x  \; \right)  \; \right]  \\
 &=& 0.
 \end{eqnarray*}
 To prove the degeneracy of the minimal graph  $\Sigma$, we exploit Lemma \ref{Gauss map for minimal graphs}. 
  Its generalized Gauss map $\mathcal{G}:\Omega \rightarrow {\mathcal{Q}}_{2} \subset {\mathbb{C}}{\mathbb{P}}^{3}$ can be
 explicitly given in terms of the coordinates $(x, y):$
\[
  \mathcal{G}(x,y) = \left[ z_1 : z_2 : z_3 : z_4 \right] = \left[  \frac{G}{\omega}: \;  i- \frac{F}{\omega}: \;  \frac{G}{\omega} f_x +  \left(  i - \frac{F}{\omega} \right)   f_y : \;  \frac{G}{\omega} g_x + \left(  i - \frac{F}{\omega} \right)  g_y \right],  
\]
 The Osserman systems (\ref{OSS definition 1}) and (\ref{OSS definition 2}) yield
 \[
\left(\, f_{y}, \, g_{y} \, \right) = \left(\, \mu \left( \frac{F}{\omega} g_{y} -\frac{G}{\omega} g_{x} \right), \,
\frac{1}{\mu}  \left( \frac{G}{\omega} f_{x} -\frac{F}{\omega} f_{y} \right) \,  \right),
 \]
 which can be complexified to 
\begin{equation} \label{birth OSS}
\frac{G}{\omega} f_x +  \left(  i - \frac{F}{\omega} \right)   f_y = - i \mu \left( \frac{G}{\omega} g_x + \left(  i - \frac{F}{\omega} \right)  g_y \right). 
\end{equation} 
Therefore, by the above formula, its generalized Gauss map $\mathcal{G}$ lies on the hyperplane 
\[
z_{3} + i \mu z_{4} =0.
\]
\end{proof}

\section{Applications of Lagrange potential on minimal graphs in ${\mathbb{R}}^{3}$} 

It would be difficult to find explicit examples of  non-holomorphic minimal graphs in ${\mathbb{R}}^{4}$ 
by directly solving the minimal surface system of the second order. We 
provide a fundamental construction, which yields explicit examples of two dimensional minimal graphs 
in ${\mathbb{R}}^{4}$, whose height functions satisfy the Osserman system of the first order. 

\begin{lemma}[\textbf{Lagrange potential on minimal graphs in ${\mathbb{R}}^{3}$}] \label{Lagrange potential}
Let $\Omega \subset {\mathbb{R}}^{2}$ be a simply connected domain. Consider the two dimensional 
graph $\mathfrak{S}$ of the ${\mathcal{C}}^{2}$ function $p : \Omega \to \mathbb{R}:$
\[
 \mathfrak{S} =  \left\{   \left(x, y, p(x, y) \right) \in {\mathbb{R}}^{3} \, \vert \, (x,y) \in \Omega \right\}.
\]
Then, the following two statements are equivalent:
\begin{enumerate}
\item[\textbf{(a)}]  The graph $\mathfrak{S}$ is a minimal surface in ${\mathbb{R}}^{3}$.
\item[\textbf{(b)}]  There exists a function $q : \Omega \to \mathbb{R}$ satisfying the  Lagrange system
\begin{equation}  \label{Lag conjugate}
 \left(q_{x}, q_{y} \right)= \left( - \frac{ p_y }{\sqrt{1+{p_{x}}^{2}+{p_{y}}^{2} }}, \frac{ p_x }{\sqrt{1+{p_{x}}^{2}+{p_{y}}^{2} }}    \right),
\end{equation}  
and the gradient estimate 
\begin{equation}  \label{Lag potential spacelike}
  {q_{x}}^{2}+{q_{y}}^{2}<1.
\end{equation}
\end{enumerate}
\end{lemma}

\begin{proof}
The graph $\mathfrak{S}$ is minimal in ${\mathbb{R}}^{3}$ if and only if the height function $p(x, y)$ satisfies 
\begin{equation}  \label{Lag MSE}
0 =   \frac{\partial}{\partial x} \left(  \frac{ p_x }{\sqrt{1+{p_{x}}^{2}+{p_{y}}^{2} }}  \right)+  
 \frac{\partial}{\partial y} \left(  \frac{ p_y }{\sqrt{1+{p_{x}}^{2}+{p_{y}}^{2} }}  \right), \quad \left(x, y \right) \in \Omega,
\end{equation}  
which indicates that the one form
\begin{equation}  \label{Lag one form}
    \, - \frac{ p_y }{\sqrt{1+{p_{x}}^{2}+{p_{y}}^{2} }} dx + \frac{ p_x }{\sqrt{1+{p_{x}}^{2}+{p_{y}}^{2} }} dy 
\end{equation}  
is closed. Since $\Omega$ is simply connected, by Poincar\'{e} Lemma, it means that (\ref{Lag one form}) is exact:
\[
 - \frac{ p_y }{\sqrt{1+{p_{x}}^{2}+{p_{y}}^{2} }} dx + \frac{ p_x }{\sqrt{1+{p_{x}}^{2}+{p_{y}}^{2} }} dy =  dq = q_{x} dx +q_{y} dy
\]
for some potential function $q : \Omega \to \mathbb{R}$, up to an additive constant. The gradient estimate (\ref{Lag potential spacelike}) immediately 
follows from the identity $1-{q_{x}}^{2}-{q_{y}}^{2} = \frac{1}{ 1+{p_x}^{2}+{p_y}^{2}  }$.
\end{proof}
 
\begin{remark}[\textbf{Lagrange potentials and conjugate surfaces of minimal graphs in ${\mathbb{R}}^{3}$}] \label{Geometric meaning of Lagrange potential}
The exactness of the one form (\ref{Lag one form}) on the minimal graph is discovered by Lagrange \cite{Lagrange 1760}, 
who deduced the minimal surface equation (\ref{Lag MSE}). The Lagrange potential plays a critical role in the existence theory of Jenkins-Serrin 
minimal graphs \cite[Section 3]{Jenkins Serrin 1966} obtained by solving the minimal surface equation with infinite boundary value problems.
 In Theorem \ref{Lagrange deformation}, we shall use the Lagrange potential to build a deformation of minimal graphs in ${\mathbb{R}}^{3}$ defined on a domain to degenerate minimal graphs in ${\mathbb{R}}^{4}$ defined on the same domain.  As a particular case of the system (\ref{CR on graph}) in 
 Proposition \ref{CR on graph lemma}, when the graph 
 \[
 \mathfrak{S}=  \left\{ \, \left( x_{1}(x,y),  x_{2}(x,y),  x_{3}(x,y) \right)= \left(x, y, p(x, y) \right)\,  \in {\mathbb{R}}^{3} \, \vert \, (x,y) \in \Omega \right\}
 \]
is minimal in ${\mathbb{R}}^{3}$, we find that,
 for each $k \in \left\{1, 2, 3 \right\}$, the function $x_{k} + i  {x_{k}^{*}}$ is holomorphic on $\mathfrak{S}$, whenever we have the Cauchy-Riemann 
 equations on the minimal graph $\mathfrak{S}$:
\begin{equation}  \label{conjugate 1 2 3}
   \begin{bmatrix} {\left(x_{k}^{*} \right)}_x \\ {\left(x_{k}^{*}\right)}_y \end{bmatrix}  =    \begin{bmatrix}  \frac{p_x p_y}{\sqrt{1+{p_x}^{2}+{p_y}^{2}}} & - \frac{ 1+{p_x}^{2} }{\sqrt{1+{p_x}^{2}+{p_y}^{2}}} 
  \\   \frac{ 1+{p_y}^{2}}{\sqrt{1+{p_x}^{2}+{p_y}^{2}}} & -\frac{  p_x p_y  }{\sqrt{1+{p_x}^{2}+{p_y}^{2}}} \end{bmatrix}    \begin{bmatrix}   {\left(x_{k} \right)}_x \\ {\left(x_{k}\right)}_y \end{bmatrix}.
\end{equation}
\begin{enumerate}
\item[{\textbf{(a)}}]  The conjugate surface ${\mathfrak{S}}^{*}=  \left\{ \, \left( x_{1}^{*}(x,y),  x_{2}^{*}(x,y),  x_{3}^{*}(x,y) \right)\,  \in {\mathbb{R}}^{3} \, \vert \, (x,y) \in \Omega \right\}$ becomes a minimal surface locally isometric to  $\mathfrak{S}$. According to Krust 
Theorem \cite[p. 122]{Dierkes Hildebrandt Sauvigny 2010}, whenever  the domain $\Omega$ of the initial minimal graph  $\mathfrak{S}$ is \emph{convex}, its conjugate minimal surface ${\mathfrak{S}}^{*}$ can be represented as a graph on some domain ${\Omega}^{*}$. 
\item[{\textbf{(b)}}] Taking $k=3$ in (\ref{conjugate 1 2 3}) yields the Lagrange system (\ref{Lag conjugate}), which is equivalent to
\begin{equation}  \label{Lag conjugate here}
   \begin{bmatrix} q_x \\ q_y \end{bmatrix}  =    \begin{bmatrix}  \frac{p_x p_y}{\sqrt{1+{p_x}^{2}+{p_y}^{2}}} & - \frac{ 1+{p_x}^{2} }{\sqrt{1+{p_x}^{2}+{p_y}^{2}}} 
  \\   \frac{ 1+{p_y}^{2}}{\sqrt{1+{p_x}^{2}+{p_y}^{2}}} & -\frac{  p_x p_y  }{\sqrt{1+{p_x}^{2}+{p_y}^{2}}} \end{bmatrix}    \begin{bmatrix}  p_x \\ p_y \end{bmatrix},
\end{equation}
The function $p+i q$ is holomorphic on  $\mathfrak{S}$ (with respect to the classical conformal coordinates constructed in Proposition \ref{CR on graph lemma}).
   \item[{\textbf{(c)}}] We combine the gradient estimation (\ref{Lag potential spacelike}) and the Lagrange system (\ref{Lag conjugate}) to deduce
  \begin{equation}  \label{MaxSE 1}
  \frac{\partial}{\partial x} \left(  \frac{ q_x }{\sqrt{1-{q_{x}}^{2}-{q_{y}}^{2} }}  \right)+  
 \frac{\partial}{\partial y} \left(  \frac{ q_y }{\sqrt{1-{q_{x}}^{2}-{q_{y}}^{2} }}  \right) 
 =  \frac{\partial}{\partial x} \left(  -p_{y}  \right)+  \frac{\partial}{\partial y} \left(  p_{x}  \right) =0,
\end{equation}  
or equivalently, 
  \begin{equation}  \label{MaxSE 2}
 \left( 1 - {q_{y}}^{2} \right) q_{xx} + 2 q_{x} q_{y} q_{xy} +  \left( 1 - {q_{x}}^{2} \right) q_{yy}  =0.
\end{equation}  
As a historical remark, the dual equation (\ref{MaxSE 2}) is reported in 1855 by Catalan \cite[Equation (C), p. 1020]{Catalan 1855},  where he
 discovered his minimal surface generated by a one parameter family of parabolas and contains a cycloid as a geodesic. 
Calabi found that, in his article \cite{Calabi 1970} on Bernstein type problems,  (\ref{Lag potential spacelike}) and the dual 
equation (\ref{MaxSE 1}) indicates that the graph $z=q(x,y)$ is a maximal surface (spacelike surface with zero mean curvature) in 
Lorentz-Minkowski space ${\mathbb{L}}^{3}=\left({\mathbb{R}}^{3}, dx^{2}+dy^{2}-dz^{2} \right)$. The author \cite{Lee 2013} extended the Calabi
duality between minimal graphs in ${\mathbb{R}}^{3}$ and maximal graphs in ${\mathbb{L}}^{3}$ to higher codimensions.
 \item[{\textbf{(d)}}]  Taking $k=1$ and $k=2$ in (\ref{conjugate 1 2 3})  yields two identities
\begin{equation}  \label{divergence 1}
   \frac{\partial}{\partial y} \left(  \frac{p_x p_y}{\sqrt{1+{p_x}^{2}+{p_y}^{2}}}  \right) =   \frac{\partial}{\partial x} \left(  \frac{ 1+{p_y}^{2}}{\sqrt{1+{p_x}^{2}+{p_y}^{2}}}  \right),
\end{equation} 
and 
\begin{equation}  \label{divergence 2}
  \frac{\partial}{\partial y} \left(  \frac{1+{p_x}^{2}}{\sqrt{1+{p_x}^{2}+{p_y}^{2}}}  \right) =   \frac{\partial}{\partial x} \left(  \frac{  p_x p_y    }{\sqrt{1+{p_x}^{2}+{p_y}^{2}}}  \right).
\end{equation} 
Following previous notations, these two equalities can be rewritten as  
\begin{equation}
\frac{\partial}{\partial y} \left( \frac{F}{\omega} \right) = \frac{\partial}{\partial x} \left( \frac{G}{\omega} \right) \quad \text{and} \quad 
\frac{\partial}{\partial y} \left( \frac{E}{\omega} \right) = \frac{\partial}{\partial x} \left( \frac{F}{\omega} \right).
\end{equation} 
\end{enumerate}
 \end{remark}
 
 \bigskip

\begin{theorem}[\textbf{Degenerate minimal graphs in ${\mathbb{R}}^{4}$ derived from minimal graphs in ${\mathbb{R}}^{3}$}] \label{Lagrange deformation}
Let ${\Sigma}_{0}$ be the minimal graph of the ${\mathcal{C}}^{2}$ function $p : \Omega \to \mathbb{R}$ defined on a domain  $\Omega \subset {\mathbb{R}}^{2}:$
\[
{\Sigma}_{0} =  \left\{   \left(x, y, p(x, y) \right) \in {\mathbb{R}}^{3} \, \vert \, (x,y) \in \Omega \right\}.
\]
Let $q : \Omega \to \mathbb{R}$ be the Lagrange potential in Lemma \ref{Lagrange potential}, which solves the Lagrange system
\begin{equation}  \label{Lag conjugate 2}
 \left(q_{x}, q_{y} \right)= \left( - \frac{ p_y }{\sqrt{1+{p_{x}}^{2}+{p_{y}}^{2} }}, \frac{ p_x }{\sqrt{1+{p_{x}}^{2}+{p_{y}}^{2} }}    \right). 
\end{equation}  
For any constant $\lambda \in \mathbb{R}-\{0\}$, we associate the two dimensional graph  in ${\mathbb{R}}^{4}$:
 \begin{equation}  \label{Lag potential graph}
{\Sigma}_{\lambda} =\left\{\,  \left(x, y,  \left( \cosh \lambda \right) \, p(x, y),  \left( \sinh \lambda \right) \,  q(x, y) \right) \in {\mathbb{R}}^{4} \, \vert \, (x,y) \in \Omega \, \right\}.
\end{equation}  
Then, the pair $\left(  f(x,y), g(x,y)  \right) = \left(  \left( \cosh \lambda \right) p(x, y),   \left( \sinh \lambda \right) q(x, y) \right)$ solves 
the Osserman system (\ref{OSS definition 1}) with the coefficient $\mu = \coth \lambda$. In particular, the graph ${\Sigma}_{\lambda}$ is 
minimal in ${\mathbb{R}}^{4}$. Also, we obtain the conformal invariance of the conformally changed induced metric 
 \begin{equation}  \label{conformal invariance}
\frac{1}{\sqrt{\det{ \left( {\mathbf{g}}_{{\Sigma}_{\lambda}}  \right) }}}  {\mathbf{g}}_{{\Sigma}_{\lambda}} = \frac{1}{\sqrt{\det{ \left( {\mathbf{g}}_{{\Sigma}_{0}}  \right) }}}  {\mathbf{g}}_{{\Sigma}_{0}}. 
\end{equation}
\end{theorem}

\begin{proof} Our goal is to verify the Osserman system (\ref{OSS definition 1}) with the coefficient $\mu = \coth \lambda$:
 \begin{equation}    \label{OSS definition 3}
  \begin{bmatrix}   f_x \\ f_y \end{bmatrix}  =  \coth \lambda  \begin{bmatrix}  \frac{E}{\omega} & \frac{F}{\omega} 
  \\ \frac{F}{\omega} & \frac{G}{\omega} \end{bmatrix}    \begin{bmatrix}   g_y \\ -  g_x \end{bmatrix}.  
 \end{equation} 
Taking $W=\sqrt{1+{p_{x}}^{2}+{p_{y}}^{2} } \geq 1$ and using the system (\ref{Lag conjugate 2}), we have   
\begin{equation}  \label{Lag conjugate 3}
 \left(q_{x}, q_{y} \right)= \left( - \frac{ p_y }{W}, \frac{ p_x }{W}    \right) \quad \text{and} \quad {q_{x}}^{2}+{q_{y}}^{2} = \frac{W^2 -1}{W^2}. 
\end{equation}  
We use the definition $\left(  f, g )  \right) = \left(  \left( \cosh \lambda \right) p,   \left( \sinh \lambda \right) q  \right)$  to deduce  
\begin{eqnarray*} 
 {\omega}^{2} &=&  EG - F^2 \\
  &=&    \left(  1+ {f_{x}}^{2}+{g_{x}}^{2} \right) \left( 1+ {f_{y}}^{2}+{g_{y}}^{2} \right) - {\left( f_{x} f_{y}  +  g_{x} g_{y}\right)}^{2}\\
  &=&  1 + \left(  {f_{x}}^{2} +  {f_{y}}^{2} \right)  + \left(  {g_{x}}^{2} +  {g_{y}}^{2} \right) +  { \left(  f_x g_y -f_{y} g_{x} \right) }^2   \\
  &=& {  \left[  \left( \cosh^{2} \lambda \right) W  - \frac{\sinh^{2} \lambda}{W}    \right]}^{2}.
\end{eqnarray*} 
We observe that  
\[
\left( \cosh^{2} \lambda \right) W  - \frac{\sinh^{2} \lambda}{W} \geq  \left( \cosh^{2} \lambda \right) \cdot 1 -  \frac{\sinh^{2} \lambda}{1} =1>0,
\]
which implies that 
 \begin{equation}    \label{omega W}
\omega= \sqrt{EG -F^2} = \left( \cosh^{2} \lambda \right) W  - \frac{\sinh^{2} \lambda}{W}.
 \end{equation} 
We use (\ref{Lag conjugate 3}) to obtain the first row equality in (\ref{OSS definition 3}):
 \begin{eqnarray*} 
  \frac{E}{\omega} g_{y}  -  \frac{F}{\omega} g_{x}  &=&  \frac{1}{\omega} \cdot { \left[   \left(  1+ {f_{x}}^{2}+{g_{x}}^{2} \right)  g_{y} -  {\left( f_{x} f_{y}  +  g_{x} g_{y}\right)} g_x  \right]} \\
  &=&  \frac{1}{\omega} \cdot  \left[   \left(  1+ {f_{x}}^{2}  \right)  g_{y} -   f_{x} f_{y}    g_x  \right]  \\
 &=&   \frac{\sinh \lambda}{\omega} \cdot \frac{p_x}{W} \cdot \left[   1 + \left( \cosh^{2} \lambda \right) \left(   {p_{x}}^{2}+{p_{y}}^{2} \right)    \right]   \\
 &=&   \frac{\sinh \lambda}{\omega} \cdot \frac{p_x}{W} \cdot \left[   - \sinh^{2} \lambda + \left( \cosh^{2} \lambda \right) W^2   \right]    \\
  &=&  \left( \sinh \lambda \right)  p_{x} \\ 
  &=&  \frac{f_{x}}{\coth \lambda}.
\end{eqnarray*} 
 We omit a similar verification of the second row equality in (\ref{OSS definition 3}). Finally, the conformal invariance (\ref{conformal invariance}) comes from the equalities
 \[
 \left(\frac{E}{\omega}, \frac{F}{\omega}, \frac{G}{\omega} \right)=\left(\frac{1+{p_x}^{2}}{W}, \frac{p_x p_y }{W}, \frac{{1+{p_y}^{2}}}{W} \right).
 \]
\end{proof}
    
     We apply Theorem \ref{Lagrange deformation} to classical minimal graphs in ${\mathbb{R}}^{3}$ to find explicit examples of 
  old and new minimal graphs in ${\mathbb{R}}^{4}$.
   
\begin{example}[\textbf{One parameter family of minimal surfaces in ${\mathbb{R}}^{4}$ foliated by hyperbolas or lines}]  \label{Lee hyperbola} 
On the infinite strip 
$\Omega:=\left\{ \, \left(x, y \right) \in {\mathbb{R}}^{2} \; \vert \;    x \in \mathbb{R}    \; \text{and} \;  y \in   \left(- \frac{\pi}{2}, \frac{\pi}{2}  \right)     \right\}$,
we consider the fundamental piece of the \textbf{helicoid} in ${\mathbb{R}}^{3}$:
\[
{\Sigma}_{0} =  \left\{   \left(x, y,  x \tan y \right) \in {\mathbb{R}}^{3} \, \vert \, (x,y) \in \Omega \right\}.
\]
Solving the induced Lagrange system (\ref{Lag conjugate}) in Lemma \ref{Lagrange potential} 
\[
 \left(q_{y}, -q_{x} \right)= \left( \frac{ p_x }{\sqrt{1+{p_{x}}^{2}+{p_{y}}^{2} }}, \frac{ p_y }{\sqrt{1+{p_{x}}^{2}+{p_{y}}^{2} }}    \right)=
 \left(\frac{\cos y \sin y }{ \sqrt{ {\cos}^2 y +x^2 \, } }, \frac{x}{ \sqrt{ {\cos}^2 y +x^2 \, } } \right),
\]
we obtain $p(x, y) = x \tan y$$q(x,y)= -\sqrt{ {\cos}^2 y +x^2 \, }$, up to an additive constant.  
Applying Theorem \ref{Lagrange deformation} with any constant $\lambda \in \mathbb{R}$, we obtain the solution of the Osserman system:
\[
\left( f(x,y), g(x,y) \right) = \left(  \left(  \cosh \lambda \right)\,  x \tan y, - \sinh \lambda  \sqrt{ {\cos}^2 y +x^2 \, }  \right),
\]
which associates, after an reflection, the two dimensional minimal graph ${\Sigma}^{-}_{\lambda}$ 
in ${\mathbb{R}}^{4}:$
\[
{\Sigma}^{-}_{\lambda} =\left\{\left(x, y, \left( \cosh \lambda \right)\,  x \tan y, \, \sinh \lambda  \sqrt{ {\cos}^2 y +x^2 \, } \right) \in {\mathbb{R}}^{4} \, \vert \, (x,y) \in   {\Omega} \, \right\}.
\]
\begin{enumerate}
\item[\textbf{(a)}] When $\lambda=0$, the graph ${\Sigma}^{-}_{\lambda}$ recovers the \textbf{helicoid} in ${\mathbb{R}}^{3}$ foliated by lines. 
\item[\textbf{(b)}] Let $\lambda \neq 0$. We see that the graph ${\Sigma}^{-}_{\lambda}$ is foliated by a line (the level curve cut by the hyperplane $y=0$) or a hyperbola (the level curve cut by the hyperplane $y= y_{0} \in  \left(- \frac{\pi}{2}, 0 \right) \cup \left(0, \frac{\pi}{2}  \right)$).
\end{enumerate}
Under the coordinate transformation $(x, y)=\left( \sinh \mathcal{U}\cos \mathcal{V}, \mathcal{V} \right) \to \left(\mathcal{U}, \mathcal{V} \right)$,
we obtain the conformal harmonic patch for the minimal surface ${\Sigma}^{-}_{\lambda}$  in ${\mathbb{R}}^{4}$:
\begin{equation} \label{conformal harmonic theta -}
{\mathbf{F}}^{-}_{\theta}\left(\mathcal{U}, \mathcal{V}\right) 
= \left( \sinh \mathcal{U}\cos \mathcal{V}, \mathcal{V},  \cosh \lambda \cosh \mathcal{U}\cos \mathcal{V},  \sinh \lambda \sinh \mathcal{U} \sin \mathcal{V}  \right).
\end{equation} 
The graph  ${\Sigma}^{-}_{\lambda}$ belongs to the family of minimal surfaces discovered by the 
author \cite[Example 6.1]{Lee 2017}. It was originally discovered by an application of the so called 
parabolic rotations of holomorphic null curves in ${\mathbb{C}}^{3} \subset {\mathbb{C}}^{4}$ lifted from helicoids in ${\mathbb{R}}^{3}$.  
\end{example}
 
\begin{example}[\textbf{Hoffman-Osserman's minimal surfaces in ${\mathbb{R}}^{4}$ foliated by ellipses}] \label{HO ellipse} 
We consider a half of the \textbf{catenoid} in ${\mathbb{R}}^{3}$:
\[
{\Sigma}_{0} =  \left\{   \left(x, y,  \sqrt{ -x^2 + \cosh^{2} y\,}  \right) \in {\mathbb{R}}^{3} \, \vert \, (x,y) \in \Omega \right\},
\]
over the domain 
\begin{equation}
\Omega:=\left\{ \, \left(x, y \right) \in {\mathbb{R}}^{2} \; \vert \;   x^2 \leq \cosh^{2} y \,  \right\},
\end{equation}
which can be compared with the \textbf{exceptional domain} \cite[Section 7.2 and Figure 1]{Traizet 2014}. We find the pair 
$\left( p(x,y), q(x,y) \right)=\left( \sqrt{ -x^2 + \cosh^{2} y \,}, \, x \tanh y \right)$
 solving the Lagrange system 
\[
 \left(q_{y}, -q_{x} \right)= \left( \frac{ p_x }{\sqrt{1+{p_{x}}^{2}+{p_{y}}^{2} }}, \frac{ p_y }{\sqrt{1+{p_{x}}^{2}+{p_{y}}^{2} }}    \right)=
 \left(  \frac{- x}{\cosh^{2} y  }, \frac{\sinh y}{\cosh y} \right).
\]
Applying Theorem \ref{Lagrange deformation} with a constant $\lambda \in \mathbb{R}-\{0\}$,
we obtain the minimal graph ${\Sigma}^{+}_{\lambda}$ in ${\mathbb{R}}^{4}$:
\[
{\Sigma}^{+}_{\lambda} =\left\{\left(x, y,   \cosh \lambda \sqrt{ -x^2 + \cosh^{2} y\,} , \left(  \sinh \lambda \right) x \tanh y \right) \in {\mathbb{R}}^{4} \, \vert \, (x,y) \in   {\Omega}  \, \right\}.
\]
Under the coordinate transformation $(x, y)=\left( \cosh \mathcal{U}\cos \mathcal{V}, \mathcal{U} \right) \to \left(\mathcal{U}, \mathcal{V} \right)$,
we obtain the conformal harmonic patch for the minimal surface ${\Sigma}^{+}_{\lambda}$  in ${\mathbb{R}}^{4}$:
\begin{equation} \label{conformal harmonic theta +}
{\mathbf{F}}^{+}_{\lambda}\left(\mathcal{U}, \mathcal{V}\right) 
= \left( \cosh \mathcal{U}\cos \mathcal{V}, \mathcal{U},  \cosh \lambda \cosh \mathcal{U}\sin \mathcal{V},  \sinh \lambda \sinh \mathcal{U} \cos \mathcal{V}  \right).
\end{equation} 
This recovers Osserman-Hoffman's minimal annuli in ${\mathbb{R}}^{4}$ with total curvature $-4\pi$ (\cite[Proposition 6.6 and Remark 1]{Hoffman Osserman 1980} and \cite[Example 6.2 and Theorem 6.3]{Lee 2017}).  
\end{example}
 
 \begin{remark}[\textbf{Holomorphic null curves lifted from degenerate minimal graphs in ${\mathbb{R}}^{4}$}]  
 In Theorem \ref{Lagrange deformation}, if the initial minimal graph ${\Sigma}_{0}$ in ${\mathbb{R}}^{3}$ is  induced by the 
 holomorphic null curve $\phi = \left({\phi}_{1}(\zeta), {\phi}_{3}(\zeta), {\phi}_{3}(\zeta) \right)$  in ${\mathbb{C}}^{3}$
  with  ${{\phi}_{1}}^{2}+{{\phi}_{2}}^{2}+{{\phi}_{3}}^{2}=0$
   and a local conformal coordinate $\zeta$ on ${\Sigma}_{0}$,  the minimal graph ${\Sigma}_{\lambda}$ in ${\mathbb{R}}^{4}$ is induced by
 \[
 \left({\phi}_{1}(\zeta), {\phi}_{2}(\zeta), \left( \cosh \lambda  \right) {\phi}_{3}(\zeta),  \left( - i \sinh \lambda  \right) {\phi}_{3}(\zeta) \right)
 \]
 with the conformal coordinate $\zeta$ on  ${\Sigma}_{\lambda}$. Indeed, we see that the identity 
\[
\left( \cosh \lambda \right) ^{2} + \left( - i \sinh \lambda \right) ^{2} =1
\]
guarantees the nullity of the induced holomorphic curve in ${\mathbb{C}}^{4}$
\[
{{\phi}_{1}}^{2}+{{\phi}_{2}}^{2}+{\left[ \left( \cosh \lambda  \right) {\phi}_{3} \right]}^{2} +{\left[ \left(- i \sinh \lambda  \right) {\phi}_{3} \right]}^{2}=
{{\phi}_{1}}^{2}+{{\phi}_{2}}^{2}+{{\phi}_{3}}^{2}= 0.
\] 
For a concise survey of various deformations of holomorphic null curves in ${\mathbb{C}}^{n}$ lifted from minimal 
surfaces in ${\mathbb{R}}^{n}$, we invite interested readers to refer to \cite[Section 2]{Lee 2017}. 
 \end{remark}
  
 \begin{remark}[\textbf{Associate family of minimal surfaces in higher codimensions}]
 We point out that the conformal harmonic patches (\ref{conformal harmonic theta +}) in 
Example \ref{HO ellipse} and  (\ref{conformal harmonic theta -}) in Example \ref{Lee hyperbola} represent \emph{conjugate} 
minimal surfaces in ${\mathbb{R}}^{4}$. For the notion of associate family of locally isometric minimal surfaces in ${\mathbb{R}}^{n+2 \geq 3}$, we invite interested readers to refer to \cite{Lawson 1971}.
 \end{remark}

\begin{example}[\textbf{Doubly periodic minimal graphs in ${\mathbb{R}}^{4}$ derived from Scherk's doubly periodic graph in ${\mathbb{R}}^{3}$}] \label{deformed scherk doubly} 
Over the open square
$\Omega:=\left\{ \, \left(x, y \right) \in {\mathbb{R}}^{2} \; \vert \;    x,  y \in   \left(- \frac{\pi}{2}, \frac{\pi}{2}  \right)     \right\}$,
we define the fundamental piece of the doubly periodic graph in ${\mathbb{R}}^{3}$:
\begin{equation} \label{Scherk orthogonal}
 \left\{   \left(x, y,  \ln \left( \frac{\cos x}{\cos y} \right) \right) \in {\mathbb{R}}^{3} \, \vert \, (x,y) \in \Omega \right\},
\end{equation}  
which was originally discovered by Scherk \cite[p. 196]{Scherk 1835}. See \cite{Lee 2017 slides} for the picture of the surface.
Theorem \ref{Lagrange deformation} with a constant $\lambda \in \mathbb{R}-\{0\}$ gives the minimal graph 
 in ${\mathbb{R}}^{4}$:
\[
\left\{\left(x, y,  \left(  \cosh \lambda \right) \ln \left( \frac{\cos x}{\cos y} \right) ,
 \left(  \sinh \lambda \right) \arcsin\left( \sin x \sin y \right) \right) \in {\mathbb{R}}^{4} \, \vert \, (x,y) \in   {\Omega} \, \right\}.
\]
\end{example}

\begin{remark}[\textbf{Finn-Osserman proof of Bernstein's Theorem, \cite{Finn Osserman 1984}}] \label{Finn Osserman proof}
 We point out that Finn and Osserman used Scherk's first surface to present an elegant geometric proof of 
 Bernstein's Theorem that they only entire minimal graphs in ${\mathbb{R}}^{3}$ are planes.
\end{remark}

\begin{remark}[\textbf{Jenkins-Serrin type minimal graphs}]
Inspired by the existence of Scherk's first surfaces, Jenkins and Serrin \cite{Jenkins Serrin 1966} offers a powerful analytic method 
to extend Scherk's construction. The fundamental piece of Scherk's first surface can be obtained as 
a Jenkins-Serrin graph by solving the Dirichlet problem for the minimal surface equation over a square and taking boundary 
values plus infinity on two opposite sides and minus infinity on the other two opposite sides. See also Karcher's beautiful 
paper \cite[Section 2.6.1]{Karcher 1988}.
\end{remark}

\begin{remark}[\textbf{Lagrangian Scherk graph in ${\mathbb{R}}^{4}$,   \cite[Example 3]{Lee 2013}}]
 It would be interesting to develop the Jenkins-Serrin type construction for minimal surface system. 
We consider  
 \[
\Sigma = \left\{\left(x, y,  \textrm{arsinh}  \left( \tan x \cos y \right),  \textrm{arsinh}  \left( \tan y \cos x \right)  \right) \in {\mathbb{R}}^{4} \, \vert \, (x,y) \in   {\Omega} \, \right\},
\]
where $\Omega$ is the domain of the Scherk's graph (\ref{Scherk orthogonal}) in Example \ref{deformed scherk doubly}. Due to the equality 
\[
\frac{\partial}{\partial y} \left( \textrm{arsinh}  \left( \tan x \cos y \right) \right) =  \frac{\partial}{\partial x} \left( \textrm{arsinh}  \left( \tan y \cos x \right) \right), 
\]
we find that the minimal surface $\Sigma$ is a gradient graph. Taking the function $h(x,y)$ with
\[
 \left(h_{x}, h_{y} \right) =\left(\textrm{arsinh}  \left( \tan x \cos y \right),   \textrm{arsinh}  \left( \tan y \cos x \right) \right), 
\]
one deduce the Monge-Amp\'{e}re equation 
\[
h_{xx}h_{yy}-{h_{xy}}^{2}=1,
\]
a particular case of \textbf{special Lagrangian equation} with the phase (\cite{Fu 1998} and \cite[Section 1]{Warren Yuan 2009}). Since $\Sigma$ is special Lagrangian 
in ${\mathbb{C}}^{2}$, (as in \cite[Section 8.1.1]{Joyce 2007}), it is a holomorphic curve with respect to other orthogonal 
complex structure on ${\mathbb{R}}^{4}$. Taking the 
complexification 
\[
\left({\zeta}_{1}, {\zeta}_{2}\right)=\left(x+ih_{y}, y -i h_{x} \right),
\]
 we find that, $h_{xx}h_{yy}-{h_{xy}}^{2}=1$ implies that the complex curve $\Sigma$ is holomorphic:
 \[
 d{\zeta}_{1} \wedge d{\zeta}_{2} = - i \left(   h_{xx} h_{yy} -{h_{xy}}^{2} - 1  \right) dx \wedge dy = 0.
 \]
\end{remark}

\begin{example}[\textbf{Doubly periodic minimal graphs in ${\mathbb{R}}^{4}$ derived from sheared Scherk's doubly periodic graph in ${\mathbb{R}}^{3}$}] \label{deformed scherk doubly 2} 
 More generally, Scherk \cite[p. 187]{Scherk 1835} showed that his surface (\ref{Scherk orthogonal}) in 
 Example \ref{deformed scherk doubly} belongs to a one parameter family of doubly periodic minimal graphs.
Following \cite[p. 70]{Nitsche 1988}, for an angle constant $2\alpha \in \left(0, \pi \right)$ and a dilation constant $\rho > 0$, we define Scherk's doubly periodic minimal
 graph ${\Sigma}_{\rho}^{2\alpha}$ by   
\begin{equation} \label{Scherk general angle}
  z = \frac{1}{\rho}  \ln \, \left[ \frac{\cos \left( \frac{\rho}{2}   \left[  \frac{x}{\cos \alpha} -  \frac{y}{\sin \alpha}  \right] \right)}{ \cos \left( \frac{\rho}{2}   \left[  \frac{x}{\cos \alpha} + \frac{y}{\sin \alpha} \right] \right)} \right],
\end{equation} 
 where its domain is an infinite chess board-like net of rhomboids $\Omega=\cup_{i,j\in \mathbb{Z}} {\mathcal{R}}_{ij}$. Here, we define the rhomboid domain ${\mathcal{R}}_{ij}$ with the length $\frac{\pi}{\rho}$ as follows  
 \[
{\mathcal{R}}_{ij}=\left\{ (x,y) \in {\mathbb{R}}^{2} :
 \left\vert    \frac{x}{\cos \alpha} -  \frac{y}{\sin \alpha} - \frac{4i}{\rho} \pi \right\vert<\frac{\pi}{\rho}, \;
 \left\vert    \frac{x}{\cos \alpha} +  \frac{y}{\sin \alpha} - \frac{4j}{\rho} \pi \right\vert<\frac{\pi}{\rho} \right\}.
\]
Taking $\alpha=\frac{\pi}{4}$ and $\rho=2$ in (\ref{Scherk general angle}), the graph ${\Sigma}_{2}^{\frac{\pi}{2}}$ is congruent to the minimal graph
(\ref{Scherk orthogonal}) in Example \ref{deformed scherk doubly}, after the $\frac{\pi}{4}$-rotation. Theorem \ref{Lagrange deformation}
gives the minimal graph in ${\mathbb{R}}^{4}$:
\[
\left\{\left(x, y, \left( \cosh \lambda \right) p(x,y),  \left( \sinh \lambda \right)  q(x,y) \right) \in {\mathbb{R}}^{4} \, \vert \, (x,y) \in   {\Omega} \, \right\},
\]
where the pair $(p(x,y), q(x,y))$ of functions is determined by  
\[
\begin{cases}
p(x,y) = \frac{1}{\rho}  \ln \, \left[ \frac{\cos \left( \frac{\rho}{2}   \left[  \frac{x}{\cos \alpha} -  \frac{y}{\sin \alpha}  \right] \right)}{ \cos \left( \frac{\rho}{2}   \left[  \frac{x}{\cos \alpha} + \frac{y}{\sin \alpha} \right] \right)} \right], \\
q(x,y)=  \frac{1}{\rho}  \arccos\left[  \cos^{2} \alpha \cos \left( \frac{\rho  x}{\cos \alpha} \right) - 
  \sin^{2} \alpha \cos \left( \frac{\rho  y}{\sin \alpha} \right) \right].
\end{cases}
\]
\end{example}

\begin{example}[\textbf{Minimal graphs in ${\mathbb{R}}^{4}$ derived from Scherk's saddle tower in ${\mathbb{R}}^{3}$}] \label{deformed scherk sinlgy} 
Over the domain
$\Omega:=\left\{ \, \left(x, y \right) \in {\mathbb{R}}^{2} \; \vert \;    -1 < \sinh x \, \sinh y <1     \right\}$,
we consider a fundamental piece of the singly periodic multi-valued graph in ${\mathbb{R}}^{3}$:
\begin{equation} \label{Scherk saddle orthogonal}
 \left\{   \left(x, y,  \arcsin \left(  \sinh x \sinh y \right) \right) \in {\mathbb{R}}^{3} \, \vert \, (x,y) \in \Omega \right\},
\end{equation}
which was originally discovered by Scherk \cite[p. 198]{Scherk 1835}. See \cite{Lee 2017 slides} for the picture of this surface. 
Theorem \ref{Lagrange deformation} with a constant $\lambda \in \mathbb{R}-\{0\}$ gives the minimal graph in ${\mathbb{R}}^{4}$:
\[
\left\{\left(x, y,  \left(  \cosh \lambda \right)  \arcsin \left(  \sinh x \sinh y \right),
 \left(  \sinh \lambda \right) \ln \left( \frac{\cosh x}{\cosh y} \right) \right) \in {\mathbb{R}}^{4} \, \vert \, (x,y) \in   {\Omega} \, \right\}.
\] 
\end{example}

\begin{remark}[\textbf{Scherk's saddle tower in ${\mathbb{R}}^{3}$ and its influences}]
Geometrically, Scherk's saddle tower  is a smooth minimal desingularization of
 two perpendicular vertical planes. Though it was discovered more than $180$ years ago, 
 it still plays a fundamental role in the modern theory of desingularizations and gluing construction for surfaces 
 with constant mean curvature and solitons to various curvature flows. Inspired by Scherk's saddle tower, 
 Karcher \cite{Karcher 1988} discovered embedded minimal surfaces in ${\mathbb{R}}^{3}$ derived from Scherk's 
 examples, with explicit Weierstrass datum, and Pacard \cite{Pacard 2002} showed the existence of ($N-2$)-periodic 
 embedded minimal hypersurfaces in ${\mathbb{R}}^{N \geq 4}$ with four hyperplanar ends.
\end{remark}

\begin{example}[\textbf{Minimal graphs  in ${\mathbb{R}}^{4}$ derived from Scherk's generalized tower in ${\mathbb{R}}^{3}$}]   \label{deformed scherk sinlgy 2} 
We take the fundamental piece of the singly periodic multi-valued graph in ${\mathbb{R}}^{3}$:
\begin{equation} \label{Scherk saddle general angle}
 z=p(x,y) = \frac{1}{\rho} \arccos \left[  \cos^{2} \alpha \cosh \left( \frac{\rho  x}{\cos \alpha} \right) - 
  \sin^{2} \alpha \cosh \left( \frac{\rho  y}{\sin \alpha} \right)    \right] 
\end{equation}
following \cite[Section 1]{Pacard 2002} and \cite[p. 74]{Nitsche 1988}.
Theorem \ref{Lagrange deformation} produces the minimal graph in ${\mathbb{R}}^{4}$ 
\[
\left\{\left(x, y, \left( \cosh \lambda \right) p(x,y),  \left( \sinh \lambda \right)  q(x,y) \right) \in {\mathbb{R}}^{4} \, \vert \, (x,y) \in   {\Omega} \, \right\},
\]
where we take the function 
\[
q(x,y)= \frac{1}{\rho}  \ln \, \left[ \frac{\cosh \left( \frac{\rho}{2}   \left[  \frac{x}{\cos \alpha} -  \frac{y}{\sin \alpha}  \right] \right)}{ \cosh \left( \frac{\rho}{2}   \left[  \frac{x}{\cos \alpha} + \frac{y}{\sin \alpha} \right] \right)} \right].
\]
\end{example}

\section{Minimal graphs in ${\mathbb{R}}^{3}$ and special Lagrangian graphs in ${\mathbb{R}}^{6}={\mathbb{C}}^{3}$}

Inspired by Bernstein's Theorem that the entire minimal graphs in ${\mathbb{R}}^{3}$ are planar, 
Fu \cite{Fu 1998}, Jost-Xin \cite{Jost Xin 2002}, Tsui-Wang \cite{Tsui Wang 2002}, Yuan \cite{Yuan 2002} established  
Bernstein type results for entire special Lagrangian graphs in even dimensional Euclidean space. 
Under suitable analytic or geometric assumptions, entire special Lagrangian graphs should be planar.
 
 \begin{theorem}[\textbf{Special Lagrangian equation in ${\mathbb{C}}^{3}$, \cite[Theorem 2. 3 and (4.8)]{Harvey Lawson 1982}}]
Let ${\mathcal{S}}$ be the gradient graph of the $\mathbb{R}$-valued  function $\mathbf{F}(x,y,z)$ in ${\mathbb{R}}^{3}$. 
Then, the $3$-fold ${\mathcal{S}}$ in ${\mathbb{R}}^{6}$ admits an orientation making it into a special Lagrangian graph in ${\mathbb{C}}^{3}$ 
under the complexification 
\[
\left(x_{1}+iy_{1}, x_{2}+iy_{2}, x_{3}+iy_{3} \right)=\left(x_{1}, x_{2}, x_{3}, y_{1}, y_{2}, y_{3} \right)
\]
when the function $\mathbf{F}(x,y,z)$ satisfies the special Lagrangian equation
\begin{equation} \label{SLE3}
 \textrm{det}   
 \begin{bmatrix}
   {\mathbf{F}}_{xx} & {\mathbf{F}}_{xy} & {\mathbf{F}}_{xz} \\
   {\mathbf{F}}_{yx} & {\mathbf{F}}_{yy} & {\mathbf{F}}_{yz} \\
   {\mathbf{F}}_{zx} & {\mathbf{F}}_{zy} & {\mathbf{F}}_{zz}. 
 \end{bmatrix} 
 =
   \textrm{tr}  \begin{bmatrix}
   {\mathbf{F}}_{xx} & {\mathbf{F}}_{xy} & {\mathbf{F}}_{xz} \\
   {\mathbf{F}}_{yx} & {\mathbf{F}}_{yy} & {\mathbf{F}}_{yz} \\
   {\mathbf{F}}_{zx} & {\mathbf{F}}_{zy} & {\mathbf{F}}_{zz}
 \end{bmatrix}. 
\end{equation}
\end{theorem}

\begin{remark}
In \cite[III.4.B. Degenerate projections and harmonic gradients]{Harvey Lawson 1982}, Harvey and Lawson investigated the interesting special case 
when the solution of the equation (\ref{SLE3}) is affine with respect to the coordinate $z$. The function  
$\mathbf{F}(x, y, z)={\mathbf{p}}(x,y)+  z \, {\mathbf{h}}(x,y)$ satisfies the special Lagrangian equation (\ref{SLE3}) if and only if the pair 
$\left({\mathbf{p}}, {\mathbf{q}} \right)$ solves the system
\[
\begin{cases}
 \left( 1 + {{\mathbf{p}}_{y}}^{2} \right)  {\mathbf{p}}_{xx} - 2 {\mathbf{p}}_{x} {\mathbf{p}}_{y}  {\mathbf{p}}_{xy} +  \left( 1 + {{\mathbf{p}}_{x}}^{2} \right)  {\mathbf{p}}_{yy}  =0, \\
 \left( 1 + {{\mathbf{p}}_{y}}^{2} \right)  {\mathbf{h}}_{xx} - 2 {\mathbf{p}}_{x} {\mathbf{p}}_{y}  {\mathbf{h}}_{xy} +  \left( 1 + {{\mathbf{p}}_{x}}^{2} \right)  {\mathbf{h}}_{yy}  =0. \\
  \end{cases}
\]
These two partial differential equations admit interesting geometric interpretations.
The first equation means that the graph of ${\mathbf{p}}(x,y)$ is a minimal surface in ${\mathbb{R}}^{3}$, and the second equation means that 
  ${\mathbf{q}}(x,y)$ is harmonic on the graph of ${\mathbf{p}}(x,y)$. In this case, the patch of the gradient graph of $\mathbf{F}(x, y, z)={\mathbf{p}}(x,y)+  z \, {\mathbf{h}}(x,y)$ is also affine with respect to $z$:
\[
 \Phi  \begin{pmatrix}  x \\ y \\ z   \end{pmatrix}  =
   \begin{bmatrix}  x \\ y \\ z \\  {\mathbf{F}}_x \\  {\mathbf{F}}_y \\ {\mathbf{F}}_{z} \end{bmatrix}  = 
    \begin{bmatrix}  x \\ y \\ z \\  {\mathbf{p}}_x (x,y) + z {\mathbf{q}}_x(x,y) \\  {\mathbf{p}}_y (x,y) +z {\mathbf{q}}_y(x,y) \\  \mathbf{q}(x,y) \end{bmatrix}  = 
        \begin{bmatrix}  x \\ y \\ 0 \\  {\mathbf{p}}_x (x,y)  \\  {\mathbf{p}}_y(x,y)  \\  \mathbf{q}(x,y) \end{bmatrix}  +
  z  \begin{bmatrix}  0 \\ 0 \\ 1 \\    {\mathbf{q}}_x (x,y)\\   {\mathbf{q}}_y(x,y) \\  0 \end{bmatrix},
\]
which geometrically means that the gradient graph is generated by lines $z \in \mathbb{R}  \mapsto \Phi \left(  x, y, z \right)$.
\end{remark}

As the second application of the Lagrange potential 
on minimal graphs in ${\mathbb{R}}^{3}$, we use the above Harvey-Lawson construction to build a one parameter family of special Lagrangian graphs in 
${\mathbb{C}}^{3}$ derived from minimal graphs in ${\mathbb{R}}^{3}$.
 
\begin{theorem}[\textbf{Lagrange potential construction of special Lagrangian graphs in ${\mathbb{R}}^{6}={\mathbb{C}}^{3}$}]  \label{to SL3}
We consider the minimal graph of the function $ \mathbf{p}(x,y) : \Omega \to \mathbb{R}$ on the domain  $\Omega \subset {\mathbb{R}}^{2}:$
\[
 \Sigma = \left\{   \left(x, y, \mathbf{p}(x, y) \right) \in {\mathbb{R}}^{3} \, \vert \, (x,y) \in \Omega \right\}.
\]
Let ${\mathbf{q}}(x,y) : \Omega \to \mathbb{R}$ be the Lagrange potential, which satisfies the Lagrange system
\[  
 \left({\mathbf{q}}_{x}, {\mathbf{q}}_{y} \right)= \left( - \frac{ {\mathbf{p}}_y }{\sqrt{1+{{\mathbf{p}}_{x}}^{2}+{{\mathbf{p}}_{y}}^{2} }}, \frac{ {\mathbf{p}}_x }{\sqrt{1+{{\mathbf{p}}_{x}}^{2}+{{\mathbf{p}}_{y}}^{2} }}    \right).
\]
For any constant $\lambda \in \mathbb{R}$, we associate the three dimensional graph ${\Sigma}_{\lambda}$ in ${\mathbb{R}}^{6}$ 
\[
 \left\{\, \left(x,\, y, \, z,\, {\mathbf{p}}_{x} - \lambda z \frac{ {\mathbf{p}}_{y} }{\sqrt{1+{{\mathbf{p}}_{x}}^{2}+{{\mathbf{p}}_{y}}^{2}}}, \,
 {\mathbf{p}}_{y} +\lambda z  \frac{ {\mathbf{p}}_{x} }{\sqrt{1+{{\mathbf{p}}_{x}}^{2}+{{\mathbf{p}}_{y}}^{2}}}, \, \lambda {\mathbf{q}}\, \right) \in {\mathbb{R}}^{6} \; \vert \; (x,y) \in \Omega, \, z \in \mathbb{R} \right\}.
\]
Then, the $3$-fold ${\Sigma}_{\lambda}$ becomes a special Lagrangian graph in ${\mathbb{C}}^{3}$.
\end{theorem}

\begin{proof} 
By the item \textbf{(c)} in Remark \ref{Geometric meaning of Lagrange potential}, the function ${\mathbf{p}} + i {\mathbf{q}}$ is holomorphic 
on the minimal graph $\Sigma$. Since ${\mathbf{p}}$ and ${\mathbf{q}}$ are harmonic on the minimal graph $\Sigma$, by Remark \ref{two equiv operators}, we obtain
\begin{equation}  \label{HL system}
\begin{cases}
 \left( 1 + {{\mathbf{p}}_{y}}^{2} \right)  {\mathbf{p}}_{xx} - 2 {\mathbf{p}}_{x} {\mathbf{p}}_{y}  {\mathbf{p}}_{xy} +  \left( 1 + {{\mathbf{p}}_{x}}^{2} \right)  {\mathbf{p}}_{yy}  =0, \\
  \left( 1 + {{\mathbf{p}}_{y}}^{2} \right)  {\left( \lambda \mathbf{q} \right)}_{xx} - 2 {\mathbf{p}}_{x} {\mathbf{p}}_{y}  {\left( \lambda \mathbf{q} \right)}_{xy} +  \left( 1 + {{\mathbf{p}}_{x}}^{2} \right)   {\left( \lambda \mathbf{q} \right)}_{yy}  =0.
 \end{cases}
\end{equation}  
However, by \cite[Theorem 4.9]{Harvey Lawson 1982}, the system (\ref{HL system}) guarantees that the gradient graph of the function  $\mathbf{F}(x, y, z)={\mathbf{p}}(x,y)+ \lambda z \, {\mathbf{q}}(x,y)$ becomes a special Lagrangian $3$-fold in ${\mathbb{C}}^{3}$. 
\end{proof}

\begin{remark} 
The construction in Theorem \ref{to SL3} with $\lambda = 0$ recovers the product of the line $\mathbb{R}$ and the special Lagrangian $2$-fold given by the gradient graph of the function ${\mathbf{p}}(x,y)$. 
\end{remark}

\begin{example}[\textbf{Doubly periodic special Lagrangian graph in ${\mathbb{C}}^{3}$}]  \label{two periodic SL3}
We apply Theorem \ref{to SL3} to the fundamental piece of Scherk's doubly periodic graph on the domain $\Omega$ in Example \ref{deformed scherk doubly} to have the one parameter family of special Lagrangian graph ${\Sigma}_{\lambda}$ in ${\mathbb{C}}^{3}$:
\[
{\Sigma}_{\lambda} = \left\{\, \left(x,\, y, \, z,\,  \mathbf{A} \left(x, y, z\right), \mathbf{B}\left(x, y, z\right), \mathbf{C} \left(x, y, z\right) \right) \in {\mathbb{R}}^{6} \; \vert \; (x,y) \in \Omega, \, z \in \mathbb{R} \right\},
\]
where $\left(\mathbf{A}, \mathbf{B}, \mathbf{C} \right) = \left(  -\frac{\sin x}{\cos x} + \lambda z \frac{\sin x \cos y}{\sqrt{1-{\sin}^{2} x \; {\sin}^{2}y \,}\,} , \, 
 \frac{\sin y}{\cos y} + \lambda z \frac{\cos x \sin y}{\sqrt{1-{\sin}^{2} x \; {\sin}^{2}y\, }\,}, \, \lambda \arcsin\left( \sin x \sin y \right)  \right)$. 
\end{example}
 
\bigskip 
\bigskip  

 \textbf{Acknowledgements. }
\quad The author met Finn-Osserman's elegant proof \cite{Finn Osserman 1984} of Bernstein's Theorem in ${\mathbb{R}}^{3}$, which uses
Scherk's doubly periodic minimal graph (Remark \ref{Finn Osserman proof}), during a lecture presented by Antonio 
Ros at the international conference held in Seville, Spain. 
Part of this work was carried out while he was visiting Jaigyoung Choe at Korea Institute for Advanced Study, Seoul, Korea. 
He would like to thank KIAS for the hospitality. He would like to warmly thank J. C. (again) who presented him 
with a copy of ``\emph{A Survey of Minimal Surfaces}'' when he was an $\epsilon$.

\bigskip
\bigskip
\bigskip

\end{document}